\documentclass[12pt]{amsart}
\usepackage{amssymb,amsmath,amsfonts,latexsym}
\usepackage{bm}
\usepackage{charter}
\usepackage{array,graphics,color}
\usepackage{comment}
\usepackage{bm}
\usepackage{tikz-cd}
\numberwithin{equation}{section}
\usepackage[dvipsnames]{xcolor}
\usepackage{tikz,animate}
\usepackage{url}
\usepackage[pagebackref,hypertexnames=false, colorlinks, citecolor=BrickRed,linkcolor=blue, urlcolor=black]{hyperref}
\usepackage{etoolbox}

\makeatletter
\let\oldcontentsline\contentsline
\renewcommand{\contentsline}[4]{%
  \begingroup
  \hypersetup{linkcolor=black}%
  \oldcontentsline{#1}{#2}{#3}{#4}%
  \endgroup
}
\makeatother
\newtheorem{theorem}{Theorem}[section]
\newtheorem{lemma}[theorem]{Lemma}
\newtheorem{proposition}[theorem]{Proposition}

\newtheorem{setup}[theorem]{Setup}
\newtheorem{example}[theorem]{Example}
\newtheorem{definition}[theorem]{Definition}
\newtheorem{question}[theorem]{Question}
\newtheorem{corollary}[theorem]{Corollary}
\newtheorem{remark}[theorem]{Remark}

\newtheorem{hypothesis}[theorem]{Hypothesis}

\newcommand{\bT}{\mathbb{T}}

\newcommand{\fH}{{\mathfrak{H}}}

\newcommand{\fM}{{\mathfrak{M}}}

\newcommand{\fQ}{{\mathfrak{Q}}}

\newcommand{\bea}{\begin{eqnarray}}
\newcommand{\eea}{\end{eqnarray}}

\newcommand{\cla}{\mathcal{A}}
\newcommand{\clb}{\mathcal{B}}

\newcommand{\cld}{\mathcal{D}}

\newcommand{\clf}{\mathcal{F}}

\newcommand{\clh}{\mathcal{H}}
\newcommand{\clk}{\mathcal{K}}

\newcommand{\clm}{\mathcal{M}}
\newcommand{\cln}{\mathcal{N}}

\newcommand{\clq}{\mathcal{Q}}

\newcommand{\cls}{\mathcal{S}}

\newcommand{\clw}{\mathcal{W}}

\newcommand{\D}{\mathbb{D}}
\newcommand{\N}{\mathbb{N}}
\newcommand{\C}{\mathbb{C}}
\newcommand{\Z}{\mathbb{Z}}
\newcommand{\z}{\bm{z}}

\newcommand{\be}{\begin{equation}}
\newcommand{\ee}{\end{equation}}
\newcommand{\ben}{\begin{eqnarray*}}
\newcommand{\een}{\end{eqnarray*}}

\newcommand{\bi}{\begin{itemize}}
\newcommand{\ei}{\end{itemize}}

\newcommand\la{{\langle }}
\newcommand\ra{{\rangle}}

\begin{document}

\title[Doubly commuting invariant and co-invariant subspaces]{Doubly commuting invariant and co-invariant subspaces}

\author[Barik]{Sibaprasad Barik}
\address{Statistics and Mathematics Unit, Indian Statistical Institute, 8th Mile, Mysore Road, Bengaluru, 560059, India}

\email{\href{mailto:barik\_if@isibang.ac.in}{barik\_if@isibang.ac.in}, sibaprasadbarik00@gmail.com}

\author[Jain]{Shubham Jain}
\address{Statistics and Mathematics Unit, Indian Statistical Institute, 8th Mile, Mysore Road, Bengaluru, 560059, India} 
\email{\href{mailto:shubhamjain_pd@isibang.ac.in}{shubhamjain\_pd@isibang.ac.in}, shubhamjain4343@gmail.com}

	\author[Panja]{Samir Panja}
\address{Statistics and Mathematics Unit, Indian Statistical Institute, 8th Mile, Mysore Road, Bengaluru, 560059, India} 
\email{\href{mailto:samirpanja\_pd@isibang.ac.in}{samirpanja\_pd@isibang.ac.in}, panjasamir2020@gmail.com}



\begin{abstract}
Given complex separable Hilbert spaces $\mathcal{H}_{i}$ and bounded linear operators $T_{i}$ on $\mathcal{H}_{i}$, for $i=1,\dots,n$, we consider the Hilbert tensor product space $\mathfrak{H}:=\mathcal{H}_{1}\otimes\dots\otimes\mathcal{H}_{n}$ and the operator tuple $(\tilde{T}_{1},\dots,\tilde{T}_{n})$ on $\mathfrak{H}$, where
$$\tilde{T}_{i}:=I_{\mathcal{H}_{1}}\otimes\dots\otimes I_{\mathcal{H}_{i-1}}\otimes T_{i}\otimes I_{\mathcal{H}_{i+1}}\otimes\dots\otimes I_{\mathcal{H}_{n}}.$$
A closed jointly $(\tilde{T}_{1},\dots,\tilde{T}_{n})$-invariant subspace $\mathfrak{M}$ of $\mathfrak{H}$ is said to be doubly commuting if the restricted tuple $(\tilde{T}_{1}|_{\mathfrak{M}},\dots,\tilde{T}_{n}|_{\mathfrak{M}})$ is doubly commuting. In this paper, we establish that if the hereditary polynomials in each $T_i$, $i=1,\ldots,n$, are ultraweakly dense in $\mathcal{B}(\mathcal{H}_i)$, then a closed jointly $(\tilde{T}_1,\ldots,\tilde{T}_n)$-invariant subspace $\mathfrak{M}\subseteq\mathfrak{H}$ is doubly commuting if and only if it admits the tensor product decomposition
\[
\mathfrak{M}=\mathcal{M}_1\otimes\cdots\otimes\mathcal{M}_n,
\]
where each $\mathcal{M}_i\subseteq\mathcal{H}_i$ is a closed $T_i$-invariant subspace.

As applications, our results establish a unified framework for characterizing doubly commuting quotient modules of weighted Bergman spaces over $\mathbb{D}^n$ and tensor products of complete Nevanlinna-Pick spaces. Furthermore, this framework yields characterizations of doubly commuting jointly invariant and co-invariant subspaces for tuples of backward shift operators acting on tensor products of de Branges-Rovnyak spaces $H(b)$, where $b$ ranges over both extreme points and a class of nonextreme points. We also characterize doubly commuting jointly invariant and co-invariant subspaces for tuples of coordinate Volterra operators acting on $L^2((0,1)^n;\mathbb C)$. In some particular cases, our results recover the known characterizations of doubly commuting quotient modules of $H^2(\mathbb{D}^n)$ and doubly commuting submodules of its Jordan blocks. Moreover, the framework accommodates distinct factor spaces, thereby extending these characterizations to mixed tensor products involving arbitrary combinations of the aforementioned function spaces.
\end{abstract}

\maketitle 
\tableofcontents
\section{Introduction and statement of the main theorem}

The study of invariant subspaces of operator tuples is a fundamental problem in operator theory, with the structure of such subspaces often reflecting important algebraic and geometric properties of the underlying operators. In particular, when the components of a tuple are commuting or doubly commuting, additional structure may emerge, allowing one to obtain more precise descriptions of their invariant and co-invariant subspaces. A particularly rich class of examples arises from tuples of shift operators acting on spaces of holomorphic functions. A well studied classical space of holomorphic functions is the Hardy space over the open unit disc $\D$, denoted by $H^2(\D)$, defined as the set of all holomorphic functions $f$ on $\D$ such that coefficients of the power series of $f$ are square summable, that is,
\[
H^2(\D)=\Big\{\sum_{k=0}^{\infty}a_kz^k:a_k\in\C\,\text{and}\,\sum_{k=0}^{\infty}|a_k|^2<\infty\Big\}.
\]    
The operator $M_{z}$ on $H^2(\D)$ is the well-known shift operator defined as
\[
M_{z}f(w)=wf(w)\quad(f\in H^2(\D),\, w\in\D). 
\]
The invariant subspace theory of  $M_z$ on $H^2(\D)$, initiated by Beurling \cite{B}, has motivated a wide range of investigations into the structure of invariant subspaces for multivariable operator tuples. The study of various aspects of these problems for tuples associated with spaces of holomorphic functions in several variables was pioneered by Ahern, Douglas, Clark, and Yang \cite{ACD,AC,DY98,DY00} and has continued to develop in more recent works by Sarkar, Das, Chattopadhyay, Guo, Izuchi, Nakazi, Seto and many more; see \cite{SSS,CDS,GY,INS,S,SSW,Y}. These investigations reveal that the structure of invariant subspaces can depend significantly on both the underlying operator tuple and the ambient Hilbert space, thereby motivating the study of structural characterizations under suitable assumptions on the tuple.

Our investigation of the structure of invariant subspaces focuses on operator tuples acting on Hilbert spaces that admit a tensor product structure. More precisely, we work within the following general framework:
\begin{setup}\label{setup}
	Let $\mathcal{H}_i$ be a complex separable Hilbert space, and let $T_i$ be a bounded linear operator on $\mathcal{H}_i$, for each $i=1,\ldots,n$. We consider the tensor product Hilbert space
\[
\mathcal{H}:=\mathcal{H}_1\otimes\cdots\otimes\mathcal{H}_n.
\]
For each $i=1,\ldots,n$, we define the amplification of $T_i$ to $\mathcal{H}$, denoted by $\widetilde{T}_i$, as
\[
\widetilde{T}_i:=
I_{\mathcal{H}_1}\otimes\cdots\otimes
I_{\mathcal{H}_{i-1}}\otimes
T_i\otimes
I_{\mathcal{H}_{i+1}}\otimes\cdots\otimes
I_{\mathcal{H}_n}.
\]
Throughout the paper, we assume that $n\geq 2$.
\end{setup}
A closed subspace $\fM\subseteq\fH$ is called jointly $(\tilde T_1,\ldots,\tilde T_n)$-invariant if $\tilde T_i\fM\subseteq\fM$ for every
$i=1,\ldots,n$, and is called jointly co-invariant if it is jointly
invariant under $(\tilde T_1^*,\ldots,\tilde T_n^*)$. A jointly $(\tilde T_1,\ldots,\tilde T_n)$-invariant closed subspace $\fM\subseteq\fH$ is said to be \textit{doubly commuting} if $(\tilde T_1|_{\fM},\ldots,\tilde T_n|_{\fM})$ on $\fM$ is doubly commuting, that is,
\[
(\tilde T_i|_{\fM})(\tilde T_j|_{\fM})=(\tilde T_j|_{\fM})(\tilde T_i|_{\fM})\quad \text{and}\quad 
(\tilde T_i|_{\fM})(\tilde T_j|_{\fM})^*=(\tilde T_j|_{\fM})^*(\tilde T_i|_{\fM}),
\] 
for all $i,j=1,\ldots,n$ and $i\neq j$. Finally, we say that a closed jointly $(\tilde T_1,\ldots,\tilde T_n)$-invariant subspace $\fM\subseteq\fH$ is \textit{hereditary} if $$\fM=\clm_1\otimes\cdots\otimes\clm_n,$$ for some $T_i$-invariant closed subspaces $\clm_i$ of $\clh_i$, for each $i=1,\ldots,n$. 

 A fundamental example of a function space that has the tensor product form is the Hardy space $H^2(\D^n)$ over the open unit polydisc $\D^n$, defined as
\[
H^2(\D^n):=\Bigg\{f=\sum_{\bm k\in\Z_+^n}a_{\bm k}\z^{\bm k}: a_{\bm k}\in\C \,\,\text{and}\,\,\sum_{\bm k\in\Z_+^n}|a_{\bm k}|^2<\infty\Bigg\},
\]
where $\Z_+^n$ is the semigroup of $n$-tuples of nonnegative integers and for $\bm z=(z_1,\ldots,z_n)\in\D^n$ and $\bm k=(k_1,\ldots,k_n)\in \Z_+^n$, $\bm z^{\bm k}:=z_1^{k_1}\cdots z_n^{k_n}$. The space $H^2(\D^n)$ can be identified with 
\[
H^2(\D)\otimes\cdots\otimes H^2(\D) \quad (n\text{-copies})
\]
via the natural unitary operator which sends the monomials $\z^{\bm k}$ to $z_1^{k_1}\otimes \cdots \otimes z_n^{k_n}$, for all $\bm k=(k_1,\ldots,k_n)\in\mathbb Z_{+}^n$. The tuple $(M_{z_1},\ldots, M_{z_n})$ on $H^2(\D^n)$ is the tuple of coordinate shift operators on $H^2(\D^n)$, where for all $i=1,\ldots,n$,
\begin{align} \label{mzsetup}
  M_{z_i}\,\,\text{on}\,\, H^2(\D^n)\cong \big(I_{ H^2(\D)}\otimes\cdots\otimes \underbrace{M_z}\limits_{i\textup{-th place}}\otimes\cdots \otimes I_{ H^2(\D)}\big)\,\,\text{on}\,\, H^2(\D)\otimes\cdots\otimes H^2(\D).  
\end{align}

A jointly $(M_{z_1},\ldots, M_{z_n})$-invariant closed subspace of $H^2(\D^n)$ is said to be a \textit{submodule of $H^2(\D^n)$}, and a jointly $(M_{z_1}^*,\ldots, M_{z_n}^*)$-invariant closed subspace is called a \textit{quotient module of $H^2(\D^n)$}. In \cite{INS, S}, the authors proved that any doubly commuting quotient module $\fQ$ of $H^2(\D^n)$ is hereditary, that is, $\fQ$ has the form
\begin{align}\label{Jordan block}
\clq_1\otimes\cdots\otimes\clq_n,
\end{align}
where each $\clq_i$ ($i=1,\ldots,n$) is a quotient module of $H^2(\D)$. This result was further generalized in \cite{CDS} in the setting of standard Hilbert modules corresponding to analytic kernels over $\D^n$. A closed subspace $\fQ$ of $H^2(\D^n)$ is said to be a \textit{Jordan block} of $H^2(\D^n)$ if it has the form as in \eqref{Jordan block}. Setting $S_i := P_{\clq_i}M_{z}|_{\clq_i}$ on $\clq_i$, where $P_{\clq_i}:H^2(\D)\to \clq_i$ is the orthogonal projection, consider the operator tuple $(\tilde{S}_1,\ldots,\tilde{S}_n)$ on the Jordan block $\fQ$ of $H^2(\D^n)$ defined by
\[
\tilde{S}_i:= I_{\clq_1}\otimes\cdots\otimes I_{ \clq_{i-1}} \otimes S_i\otimes I_{ \clq_{i+1}}\otimes\cdots \otimes I_{ \clq_n},
\]
for all $i=1,\ldots,n$. In \cite{SSS}, it is shown that any doubly commuting jointly $(\tilde{S}_1,\ldots,\tilde{S}_n)$-invariant closed subspace $\fM$ of the Jordan block $\fQ$ is also hereditary.

In each of the cases discussed above, a jointly invariant or co-invariant closed subspace $\fM$ is doubly commuting if and only if $\fM$ is hereditary. However, this equivalence does not hold in general. For instance, consider a doubly commuting submodule $\cls$ of $H^2(\D^n)$. By \cite{M}, $\cls$ is of the form $qH^2(\D^n)$ for some inner function $q$ on $\D^n$, where an inner function $q$ is a bounded analytic function on $\D^n$ such that the radial limit of $q$ to the boundary takes unimodular values on $\bT^n$, the distinguished boundary of $\D^n$. Whenever $q$ is not a variable-separating function (an example is given at the end of Section \ref{proof_MT}), $qH^2(\D^n)$ fails to be a hereditary subspace of $H^2(\D^n)$. This naturally motivates the problem of identifying classes of amplified operator tuples on $\fH$ as defined in Setup \ref{setup} for which every doubly commuting jointly invariant subspace is hereditary.

It is straightforward to verify that, under Setup \ref{setup}, every hereditary jointly $(\tilde{T}_1,\ldots,\tilde{T}_n)$-invariant closed subspace of $\fH$ is doubly commuting. This leads us to the following question concerning the structure of doubly commuting invariant subspaces of $\fH$.
 
\begin{question}\label{question1}
With the notation introduced in Setup \ref{setup}, under what conditions is every doubly commuting jointly $(\tilde{T}_1,\ldots,\tilde{T}_n)$-invariant closed subspace $\fM\subseteq\fH$ hereditary?
\end{question}

We answer the above question for a large class of operators, which we call \textit{admissible operators}, defined as follows. 

\begin{definition} \label{Defn: admissible}
    An operator $T \in \clb(\clh)$ is said to be admissible operator if $$\overline{\operatorname{span}}^{\text{UWT}}\{T^{*p} T^{q}:p,q\in \Z_+\}=\clb(\clh).$$
\end{definition}
Here, $\clb(\clh)$ denotes the space of bounded linear operators on $\clh$, and $\mathrm{UWT}$ stands for the ultraweak topology on $\clb(\clh$). Recall that a polynomial of the form $\sum_{p, q} c_{p, q}T^{*p}T^q$, where $c_{p,q}\in\C$, is called a hereditary polynomial \footnote{A word of caution: The reader should not be confused with the term \textit{hereditary}, which is used in two different contexts. In this article, our definition of \textit{hereditary} pertains to subspaces. Henceforth, we use the term \textit{hereditary} exclusively in this context.}  (see \cite{Ag85,AHS, AE}). In other words, admissible operators are those $T\in\clb(\clh)$ for which hereditary polynomials in $T$ are ultraweakly dense in $\clb(\clh)$. Every admissible operator $T$ inherently possesses the irreducibility property and, moreover, satisfies a stronger property that we term fully irreducibility, namely, $T|_{\clm}$ is irreducible for every closed $T$-invariant subspace $\clm\subseteq\clh$ (see Proposition \ref{admtostrng}). This fully irreducibility property of admissible operators plays a pivotal role in several arguments throughout our analysis. 

In the following, we state the main theorem of the paper that provides an answer to Question \ref{question1}. 
\begin{theorem} \label{mainthm}
  With the notation introduced in Setup \ref{setup}, and suppose that $T_i \in \mathcal{B}(\mathcal{H}_i)$ is an admissible operator for each $i=1,\ldots,n$. Let $\fM$ be a closed jointly $(\tilde{T}_1,\ldots,\tilde{T}_n)$-invariant subspace of $\fH$. Then $\fM$ is doubly commuting if and only if $\fM$ is hereditary, that is, there exist closed $T_i$-invariant subspaces $\clm_i \subseteq \clh_i$ ($i=1,\ldots,n$) such that
\[
\fM = \clm_1 \otimes \cdots \otimes \clm_n.
\]
\end{theorem}
We prove the above result in Section \ref{proof_MT}. In the same section, we also investigate the necessary conditions for characterizing doubly commuting closed invariant subspaces in terms of hereditary subspaces. The above result provides a unified operator-theoretic framework that recovers several known results with new proofs as well as yields new characterizations across a wide variety of function spaces discussed below:

\smallskip
\noindent \textbf{Hardy spaces and Jordan blocks:} By showing that $M_{z}^*$ on $H^2(\D)$ is admissible, we use Theorem \ref{mainthm} to immediately characterize the doubly commuting quotient modules of $H^2(\D^n)$, recovering \cite[Theorem 3.2]{S} (see Corollary \ref{hardyresultrecovery}). Similarly, by showing that the compressed shift $P_{\clq}M_{z}|_{\clq}$ on a nonzero proper quotient module $\clq\subseteq H^2(\D)$ is admissible, we apply Theorem \ref{mainthm} to characterize the doubly commuting submodules of a Jordan block of $H^2(\D^n)$, thereby recovering \cite[Theorem 5.2]{SSS} (see Corollary \ref{SSS_Theorem}).

\smallskip
\noindent \textbf{de Branges-Rovnyak spaces:} For the de Branges-Rovnyak spaces $H(b)$ associated with nonconstant extreme points $b$, we show that both $X_b:=M_z^*|_{H(b)}$ and $X_b^*$ on $H(b)$ are admissible. Consequently, Theorem \ref{mainthm} yields a characterization of doubly commuting jointly invariant subspaces, jointly co-invariant subspaces, and, more generally, jointly mixed invariant subspaces of tensor products of $H(b)$ spaces for extreme points $b$ under the amplified operators associated to $X_b$ or $X_b^*$ according to Setup \ref{setup} (see Theorem \ref{mixed}). 

For the nonextreme case, we give a complete characterization for admissibility of $X_b^*$ depending upon $b$ (see Theorem \ref{Xbstaradm}). We then provide a counterexample illustrating that $X_b$  is not admissible in general (see Lemma \ref{non-extre not adm} and Example \ref{nonextreme}). We therefore establish a sufficient condition to identify a class of nonextreme points $b$ for which $X_b$ is admissible (see Theorem \ref{non-extre adm}). Consequently, we characterize the doubly commuting invariant subspaces of the tensor product of $H(b)$ spaces under the amplified operators associated to $X_b$ or $X_b^*$ according to Setup \ref{setup} (see Theorem \ref{Xb_nonextreme_result}).

\smallskip
\noindent \textbf{Weighted Bergman and complete Nevanlinna-Pick spaces:} We establish the admissibility of the adjoint of the shift operator on weighted Bergman spaces $A^2_{\omega}(\D)$ associated with weight sequence $\omega$ satisfying Hypothesis \ref{admweightseq}. As a consequence of Theorem \ref{mainthm}, we obtain Theorem \ref{resultweightedberg}, which characterizes the doubly commuting quotient modules of weighted Bergman spaces $A^2_{\bm\omega}(\D^n)$ associated with multi-weight sequences $\bm\omega$, recovering \cite[Theorem 4.9]{CDS} as a special case (see Subsection \ref{weighted bergman}).

We also consider the reproducing kernel Hilbert space $\clh(K)$ associated with a complete Nevanlinna-Pick kernel $K$ (\cite{AglerM}) over $\D$ (satisfying \eqref{CNP_Kernel} with bounded shift $M_z$) and show that the adjoint of the shift operator on $\clh(K)$ is admissible. Thus, as an application of Theorem \ref{mainthm}, we obtain a characterization of doubly commuting quotient modules in this setting (see Subsection \ref{NPS}).

\smallskip
\noindent \textbf{Hilbert space of square-integrable complex-valued functions on $(0,1)^n$:} We prove that Volterra operator and its adjoint on $L^2((0,1);\C)$, the Hilbert space of square-integrable complex-valued functions on $(0,1)$, are admissible. Then, using Theorem \ref{mainthm} we characterize doubly commuting jointly invariant subspaces under the $n$-tuples of operators on $L^2((0,1)^n;\C)$ consists of coordinate Volterra operators and their adjoints (see Section \ref{Vol}).   

\smallskip
\noindent \textbf{Tensor product of mixed Hilbert spaces spaces:} Since setup \ref{setup} is formulated for distinct Hilbert spaces $\mathcal{H}_i$ and operators $T_i$, Theorem \ref{mainthm} is not restricted to tensor products of a single space; it naturally applies to mixed tensor products combining building blocks from any of the function space classes (e.g., Hardy spaces, de Branges-Rovnyak spaces, weighted Bergman spaces, complete Nevanlinna-Pick spaces and $L^2((0,1); \mathbb C)$), as summarized in Theorem \ref{mixedtheorem}.

\smallskip
The following notation are frequently used in this paper. Let $\mathbb Z_+$ denote the set of nonnegative integers. For a subset $X \subseteq \clb(\clh),$ $X'$ denotes the set of commutants of $X$, that is, 
$$X^\prime:=\{T \in \clb(\clh) :TS=ST, \, S \in X\}.$$
For a closed subspace $\clm$ of $\clh,$ let $P_{\clm}$ denote the orthogonal projection onto $\clm$, whereas $P_{\C \xi}$, for some vector $\xi$, denote the orthogonal projection onto the one-dimensional subspace spanned by $\xi$. The notations WOT, SOT, and UWT stands for weak operator topology, strong operator topology, and ultraweak topology in $\clb(\clh)$, respectively. For a Hilbert space $\clh,$ $I_{\clh}$ stands for the identity operator on $\clh.$ For closed subspaces $\clm_1,\clm_2\subseteq\clh$ with
$\clm_2\subseteq\clm_1$, we denote by $\clm_1\ominus\clm_2$ the orthogonal complement of $\clm_2$ in $\clm_1$. For operators $A,B\in\clb(\clh)$, we write $A\cong B$ to denote that $A$ and $B$ are unitarily equivalent. The notation $\operatorname{dim} \clh$ denotes the dimension of a Hilbert space $\clh$.

\section{Properties of admissible operators and examples}\label{Admissible}
In this section, we examine how admissibility relates to fully irreducibility, cyclicity, similarity, and defect rank. Along the way, we discuss several basic examples and non-examples. In particular, we show that normal operators and isometries on Hilbert spaces of dimension at least two are not admissible, and we also construct an admissible operator for which neither the operator nor its adjoint is cyclic. More interesting examples on concrete function spaces will appear in the later sections. The preliminary results developed here will be used both in the proof of Theorem~\ref{mainthm} in Section~\ref{proof_MT} and in the applications developed in the subsequent sections. We begin by recalling the elementary properties of ultraweak topology (see \cite{T1979}).

\begin{lemma}\label{SOT:WOT_UWT}
Let $\{T_\lambda\}$ be a net in $\clb(\clh)$ and let $T \in \clb(\clh)$. Then the following statements hold:
\begin{enumerate}
    \item $T_\lambda \to T$ in UWT if and only if
    \begin{equation}\label{def_uw}
        \sum_{j=1}^{\infty}\la T_\lambda x_j,y_j\ra
        \longrightarrow
        \sum_{j=1}^{\infty}\la Tx_j,y_j\ra
    \end{equation}
    for all sequences $(x_j)_{j\geq 1}$ and $(y_j)_{j\geq 1}$ in $\clh$ satisfying
    \[
        \sum_{j=1}^{\infty}\|x_j\|^2<\infty
        \qquad\text{and}\qquad
        \sum_{j=1}^{\infty}\|y_j\|^2<\infty.
    \]
    \item If $T_\lambda \to T$ in WOT (or SOT) and $\sup_\lambda\|T_\lambda\|<\infty$, then $T_\lambda \to T$ in UWT.
\end{enumerate}
\end{lemma}
For $T\in\clb(\clh)$, denote
\begin{align} \label{stnt}
\cls(T):=\operatorname{span}\{T^{*p} T^{q}:p,q\in \Z_+\} \quad\mbox{and} \quad \cln(T)=\overline{\cls(T)}^{\text{UWT}}.
\end{align}
With the above notation, recall that $T\in \clb(\clh)$ is admissible if $\cln(T)=\clb(\clh)$ (see Definition \ref{Defn: admissible}). 

 \subsection{Admissibility and fully irreducibility}\label{Admissibility and fully irreducibility} Recall that a closed subspace $\clm\subseteq\clh$ is said to reduce $T\in\clb(\clh)$ if $T\clm \subseteq\clm$ and $T^*\clm \subseteq\clm$. An operator $T\in\clb(\clh)$ is said to be \textit{irreducible} if the only reducing subspaces of $T$ are $\{0\}$ and $\mathcal H$. We now formally define fully irreducible operators.
\begin{definition} \label{fullyirredu}
An operator $T\in\clb(\clh)$ is said to be fully irreducible if for every closed $T$-invariant subspace $\clm\subseteq\clh$, $T|_{\mathcal M}$ is irreducible. 
\end{definition}

In the following result, we establish a key property of admissible operators. \footnote{Readers interested primarily in the proof of the main theorem may proceed directly to Section \ref{proof_MT} after Proposition \ref{admtostrng}. The remaining results and examples of this section are included for their independent interest and some of them will be used in the application sections.}
 
\begin{proposition} \label{admtostrng}
Let $T\in\clb(\clh)$ be an admissible operator. Then $T$ is fully irreducible.
\end{proposition}
\begin{proof}
 Suppose, towards a contradiction, $T$ is not fully irreducible. Then there exist a nonzero $T$-invariant closed subspace $\clm$ such that $A:=T|_{\clm}$ is
reducible. Therefore, there exists a proper closed subspace
$\{0\}\neq \clm_1\subset\clm$ such that $\clm_1$ reduces $A$. Set
$\clm_2=\clm\ominus\clm_1$. Thus
\[
\clm=\clm_1\oplus\clm_2,
\qquad \clm_1\neq\{0\},\quad\text{and}\quad\clm_2\neq\{0\}.
\]
Let $x\in\clm_1$ and $y\in\clm_2$ be two nonzero vectors. Now, for every $p,q\in\Z_+$, since $A\clm_i\subseteq\clm_i ~~(i=1,2),$ we have
\begin{equation*}
\langle T^{*p}T^q x,y\rangle =\langle T^q x,T^p y\rangle=0.
\end{equation*}
Since $T$ is admissible, the preceding observation, together with \eqref{def_uw}, yields that
\begin{align} \label{forallbh}
  \langle Bx,y\rangle=0 \qquad (B\in \clb(\clh)). 
\end{align}
Now, in particular, consider the rank-one operator $X$ defined as 
\[
Xz=\frac{\langle z,x\rangle}{\|x\|^2}\,y\qquad (z\in\clh).
\]
It is clear that $Xx=y$. Thus, by \eqref{forallbh}, we have 
\[
\|y\|^2=\langle Xx,y\rangle=0,
\]
which is a contradiction. This completes the proof.
\end{proof}

The following example illustrates that a fully irreducible operator may not be admissible, in general.
\begin{example}\label{M_z-not admi}
Consider $M_z$ on $H^2(\D)$. It is well-known that $M_z$ is fully irreducible. Now,
\begin{align*}
\cln(M_z)=\overline{\operatorname{span}}^{\,\operatorname{UWT}}\{M_z^{*p} M_z^{q}:p,q\in \Z_+\}=\overline{\operatorname{span}}^{\,\text{UWT}}\{I, M_z^{*p}, M_z^{q}:p,q\in \Z_+\}.    
\end{align*}
It is now easy to see that each operator $A \in \cln(M_z)$ satisfies 
\begin{equation*}
    M_z^*AM_z=A.
\end{equation*}
Hence, each $A \in \cln(M_z)$ is a Toeplitz operator (see \cite[Theorem 6]{BH1963}). Since the set of all Toeplitz operators on $H^2(\D)$ is not equal to $ \clb(H^2(\D))$, $M_z$ is not an admissible operator (see \cite[Corollary 4.2]{GMR}).
\end{example}
In the preceding example, $M_z$ is an isometry on $H^2(\D)$ that fails to be admissible. Surprisingly, this phenomenon is quite general: no isometry on a Hilbert space of dimension at least $2$ can be admissible. An analogous statement holds for normal operators. Before establishing this result, we prove the following simple lemma recording some properties of admissible operators which will be used in several arguments in our analysis.
\begin{lemma} \label{admtwoproperties}
Let $T \in \clb(\clh)$. Then the following statements hold:
\begin{enumerate}
    \item If $T$ is admissible, then $T|_{\mathcal M}$ is also admissible for every nonzero closed $T$-invariant subspace $\mathcal M$.
\item If $T$ is admissible and $U:\clh\to \clk$ is a unitary operator onto a Hilbert space $\clk$, then the operator $UTU^*$ is admissible on $\clk$.
\item For any $a \in \C$, $T+aI$ is admissible if and only if $T$ is admissible.
\item If $T$ is a strict contraction and $(I+T)^{-1}$ is admissible then $I+T$ is admissible.
\end{enumerate} 
\end{lemma}
\begin{proof}
  Let $\clm$ be a nonzero closed $T$-invariant subspace. It is easy to see that the compression map from $\clb(\clh)$ to
$\clb(\clm)$ is ultraweakly continuous. Moreover,
\[
P_{\mathcal M}T^{*p}T^q|_{\mathcal M}
=(T|_{\mathcal M})^{*p}(T|_{\mathcal M})^q,
\qquad p,q\in\Z_+.
\]
Thus, the admissibility of $T$ immediately yields that of
$T|_{\mathcal M}$. Hence (1) follows. (2) follows from the observation that $$(UTU^*)^{*p}(UTU^*)^{q}=UT^{*p}T^qU^* \quad (p,q\in \Z_+),$$ and $U$ is a bounded operator. Further, a simple calculation using binomial theorem shows that the linear spans of the hereditary polynomials in $T$ and $T+a I$ coincide. Consequently, $T$ is admissible if and only if $T+a I$ is
admissible, which establishes (3).

If $\|T\|<1$, the series $(I+T)^{-1}=\sum_{n=0}^{\infty}(-T)^n$ converges in norm. Consequently, every hereditary monomial in $(I+T)^{-1}$ belongs to the norm-closed linear span of hereditary monomials in $I+T$. Thus the admissibility of $(I+T)^{-1}$ implies the admissibility of $I+T$. This proves (4).
\end{proof}
\begin{proposition}\label{NorIso} Let $\clh$ be a Hilbert space. Then the following statements hold.
    \begin{enumerate}
        \item  A normal operator $T\in\clb(\clh)$ is admissible if and only if  $\clh$ is one-dimensional.

        \item An isometry $V\in\clb(\clh)$ is admissible if and only if $\clh$ is one-dimensional.
 \end{enumerate}     
\end{proposition}
\begin{proof}
    (1) The proof follows from the spectral theorem for normal operators (see \cite[Proposition 5.6]{C1990}).

    (2) Let $V$ be an isometry and admissible. If possible, let $\clh$ be not one-dimensional. If $\clh$ is finite dimensional, then $V$ is a unitary and hence a normal operator. Therefore, the conclusion follows from part (1). Now assume that $\clh$ is infinite dimensional. By the Wold decomposition (see \cite{NF}), $V$ can be written as
        \[
        V=U\oplus S\quad\text{on}\quad\clh_u\oplus\clh_s,
        \]
where $\clh_u,\clh_s\subseteq\clh$ are reducing subspaces under $V$, $U\in\clb(\clh_u)$ is a unitary and $S\in\clb(\clh_s)$ is a unilateral shift. Since $V$ is admissible, by Proposition \ref{admtostrng}, either $\clh_u=\clh$ or $\clh_s=\clh$. In the former case, $V=U$ is a unitary, and therefore, by part (1), we get a contradiction. In the latter case, $V=S$ is a unilateral shift. For a unilateral shift $S,$ if $\operatorname{rank}\,(I-SS^*)\geq 2$, $S$ is always reducible. Therefore, by Proposition \ref{admtostrng}, $\operatorname{rank}(I-SS^*)=1$. In this case, $S$ is unitarily equivalent to $M_z$ on $H^2(\D)$.
Then by Example \ref{M_z-not admi} and Lemma \ref{admtwoproperties}, $S$ is not admissible, which is a contradiction. The converse part is trivial.
\end{proof}
Interestingly, admissibility behaves quite differently under restrictions and compressions $M_z$ and $M_z^*$ on invariant and co-invariant subspaces. As an immediate consequence of the preceding result, the
restriction $M_z|_{\clm}$ is not admissible for any nonzero $M_z$-invariant subspace $\clm$ of $H^2(\D)$, since it is an isometry. Nevertheless, the compression $P_{\clm}M_z^*|_{\clm}$ is admissible. Surprisingly, for
every nonzero proper $M_z^*$-invariant subspace $\clq$ of $H^2(\D)$, both the restriction $M_z^*|_{\clq}$ and the compression $P_{\clq}M_z|_{\clq}$ are admissible (we will discuss it in Proposition \ref{extxbxbstar}). 
\subsection{Admissibility and cyclicity}\label{Admissibility and cyclicity} For nonzero vectors $x$ and $y$ in $\mathcal H,$ let $x \otimes y$ denote the rank-one operator defined as
$$ (x \otimes y)h=\langle h, y \rangle x \quad (h \in \mathcal H).$$ 
The following result provides a sufficient condition for an operator to be admissible and will be useful in the sequel. 
\begin{proposition}\label{cyclic_vector}
Let $T\in \clb(\clh)$, and let $T$ has $*$-cyclic unit vector $e$, that is,
\[\overline{\operatorname{span }}\{T^{* k} e:k\in \Z_+\}=\clh.\]
Then $T$ is an admissible operator if and only if $P_{\C e} \in \cln(T)$.
\end{proposition}
\begin{proof}
Assume that $P_{\C e}=e\otimes e \in \cln(T).$ Now, for $m, n \in \mathbb Z_+$ and $x \in \clh,$
\begin{align*}
    T^{*m}(e \otimes e)T^nx=\langle T^nx, e\rangle T^{*m}e =\langle x, T^{*n}e\rangle T^{*m}e= (T^{*m}e \otimes T^{*n}e)x.
\end{align*}
Therefore, since $\cln(T)$ is closed under left multiplication by $T^*$ and right multiplication by $T$,  $T^{*m}e \otimes T^{*n}e \in \cln(T)$ for each $m,n \in \mathbb Z_+$. Again, since $\overline{\operatorname{span}}\{T^{*k}e: k\in \mathbb Z_+\}=\clh$,  we have 
\[x\otimes y\in \cln(T)\quad (x,y\in \clh).\]
Thus, $\cln(T)$ contains every finite-rank operator. Taking finite-rank orthogonal projections $P_n \to I$ in SOT, for any $X \in \clb(\clh)$ we have $XP_n \to X$ in SOT with $\|XP_n\| \leq \|X\|$, and hence $XP_n \to X$ in UWT by Lemma~\ref{SOT:WOT_UWT}(2). Since each $XP_n \in \cln(T)$ and $\cln(T)$ is UWT-closed, it follows that $\cln(T) = \clb(\clh)$.
\end{proof}
The following is a corollary of the above result.
\begin{corollary} \label{admiffcyclic}
    Let $T \in \clb(\clh)$ be such that $I-T^*T=h \otimes h$ for some nonzero $h \in \clh.$ Then $T$ is admissible if and only if $h$ is $*$-cyclic vector for $T$.
\end{corollary}
\begin{proof}
If $\operatorname{dim} \clh= 1$, then nothing to prove. Assume that $\operatorname{dim} \clh\geq 2$.
    The sufficiency follows from Proposition \ref{cyclic_vector}. Suppose that $T$ is admissible. Assume, towards a contradiction, $h$ is not cyclic for $T^*$. Then $\mathcal M:=\overline{\operatorname{span}}\{T^{*n}h: n \in \Z_+\}^\perp$ is a nonzero closed $T$-invariant subspace of $\clh$. Observe that for any $x \in \clm$, since $h \perp \clm$,
    \begin{align*}
        \|Tx\|^2=\langle T^*Tx, x\rangle=\|x\|^2-|\langle x, h\rangle|^2=\|x\|^2.
    \end{align*}
    Thus $T|_\clm$ is an isometry. Since $T$ is admissible, by lemma \ref{admtwoproperties}, $T|_\clm$ is admissible. By Proposition \ref{NorIso}, it follows that $\clm$ is $1$-dimensional. Let $\clm=\C y$ for some unit vector $y$. Since $T|_\clm$ is an isometry, there exist $t \in \mathbb T$ such that $Ty=ty$. Since $T^*T=I-h \otimes h$ and $y\perp h$ we have $T^*y=\bar t y$. Thus $\clm$ is nonzero proper closed invariant subspace of $\clh$ that reduces $T$, which is a contradiction by Proposition \ref{admtostrng}. Thus, $h$ is a $*$-cyclic vector for $T$.
\end{proof}
We now discuss an example of an admissible operator arising immediately from the above result. We mention here that, more intriguing examples will follow in the subsequent sections.
\begin{example}\label{M_z*-admissible}
   Consider the backward shift operator $M_z^*$ on the Hardy space $H^2(\D)$. Now it is easy to see that $I-M_z M_z^*=P_{\C \bm 1}$, where $\bm 1\in H^2(\D)$ is the constant function which takes the value $1$. Also, we have
\[
\overline{\operatorname{span}}\{M_z^k\bm 1: k\in \Z_+\}=\overline{\operatorname{span}}\{z^k: k\in \Z_+\}=H^2(\D).
\]
Therefore, by Corollary \ref{admiffcyclic}, 
$M^*_z$ is an admissible operator.
\end{example}
It is natural to ask whether the converse of Proposition
\ref{cyclic_vector} holds, or whether admissibility of
$T$ implies cyclicity of $T$ itself. We show that neither
implication holds by constructing an admissible operator
$T$ for which neither $T$ nor $T^*$ is cyclic.
To this end, we first analyze the nature of admissibility under the similarity. We begin with the following result.
\begin{proposition}\label{similar_admissible}
Let $T\in\mathcal B(\mathcal H)$. Then the following statements are equivalent:
\begin{enumerate}
    \item There exists an invertible operator
    $X\in\mathcal B(\mathcal H)$ such that
    $X^{-1}TX$ is admissible.

    \item There exists a positive invertible operator
    $A\in\mathcal B(\mathcal H)$ such that
    \begin{align} \label{spanAprop}
    \overline{\operatorname{span}}^{\mathrm{UWT}}
    \{T^{*p}AT^q:p,q\in\mathbb Z_+\}
    =\mathcal B(\mathcal H).
    \end{align}
\end{enumerate}
\end{proposition}
\begin{proof}
For any invertible $X \in \mathcal B(\mathcal H)$, the operator $A := (XX^*)^{-1}$ is positive and invertible, and satisfies
\[
(X^{-1}TX)^{*p}(X^{-1}TX)^q = X^* T^{*p}(XX^*)^{-1} T^q X = X^* T^{*p} A T^q X \qquad (p,q \in \mathbb Z_+).
\]
Since $X$ is invertible, the map $Y \mapsto X^* Y X$ is a UWT-continuous linear bijection on $\mathcal B(\mathcal H)$ with UWT-continuous inverse, so $X^{-1}TX$ is admissible if and only if $A$ satisfies \eqref{spanAprop}. This immediately yields $(1)\implies(2)$. Conversely, if $A \in \mathcal B(\mathcal H)$ is positive, invertible, and satisfies \eqref{spanAprop}, then taking $X = A^{-1/2}$ gives $(XX^*)^{-1} = A$, and hence $X^{-1}TX$ is admissible.
\end{proof}

The preceding characterization leads to a useful
sufficient condition for an operator to be similar
to an admissible operator. However, the operator itself may not be admissible.

\begin{proposition}\label{similar_admissible_condition}
Let $T\in\mathcal B(\mathcal H)$. Suppose that
there exist a positive invertible operator
$C\in\mathcal B(\mathcal H)$ and a nonzero element
$u\in\mathcal H$ such that
\begin{enumerate}
    \item $T^*CT=C$. 
    \item $\displaystyle\lim_{n \rightarrow \infty }\|T^{*n}u\|= 0$.
    \item $\overline{\operatorname{span}}\{T^{*n}u : n \in \Z_+\}=\clh.$
\end{enumerate}
Then $T$ is similar to an admissible operator.
\end{proposition}

\begin{proof}
Set $A:=C+u\otimes u$. Clearly, $A\geq C.$ Thus $A$ is positive and invertible. We claim that
\begin{align*}
\overline{\operatorname{span}}^{\mathrm{UWT}}
\{T^{*p}AT^q:p,q\in\mathbb Z_+\}
=\mathcal B(\mathcal H).
\end{align*}
Indeed, since $T^*CT=C$, for all
$p,q,k\in\mathbb Z_+$, we have
\begin{align} \label{differeqn}
T^{*(p+k)}AT^{q+k}
&=T^{*(p+k)}CT^{q+k}+
T^{*(p+k)}u\otimes T^{*(q+k)}u\\
&=T^{*p}CT^q+
T^{*(p+k)}u\otimes T^{*(q+k)}u\notag.
\end{align}
Thus, by using $\lim_{n \rightarrow \infty }\|T^{*n}u\|= 0$, we have 
\begin{align*}
\lim_{k \rightarrow \infty }\|T^{*(p+k)}AT^{q+k}- T^{*p}CT^q\|=0.
\end{align*}
Consequently, for all $p,q \in \Z_+$,
\begin{align*}
T^{*p}CT^q\in
\overline{\operatorname{span}}^{\mathrm{UWT}}
\{T^{*m}AT^n:m,n\in\mathbb Z_+\}.
\end{align*}
This combined with $T^{*p}AT^q\in \operatorname{span}
\{T^{*m}AT^n:m,n\in\mathbb Z_+\}$, using \eqref{differeqn}, yields that
\begin{align*}
T^{*p}AT^q-T^{*p}CT^q
=T^{*p}u\otimes T^{*q}u \in \operatorname{span}
\{T^{*m}AT^n:m,n\in\mathbb Z_+\}\qquad(p,q\in\mathbb Z_+).
\end{align*}
Now, since $\overline{\operatorname{span}}
\{T^{*n}u:n\in\mathbb Z_+\}
=\mathcal H,$
$$x \otimes y \in \overline{\operatorname{span}}^{\mathrm{UWT}}
\{T^{*p}AT^q:p,q\in\mathbb Z_+\} \qquad (x, y \in \clh).$$
Since the linear span of the finite-rank
operators is ultraweakly dense in $\mathcal B(\mathcal H)$, we have
\begin{align*}
\overline{\operatorname{span}}^{\mathrm{UWT}}
\{T^{*p}AT^q:p,q\in\mathbb Z_+\}
=\mathcal B(\mathcal H).
\end{align*}
By Proposition \ref{similar_admissible},
$T$ is similar to an admissible operator.
\end{proof}
Applying the preceding proposition with $C=I$ to a pure isometry immediately yields the following corollary.
\begin{corollary}\label{pure_isometry_admissible}
Let $T\in\mathcal B(\mathcal H)$ be a pure isometry.
If $T^*$ has a cyclic vector, then $T$ is similar
to an admissible operator.
\end{corollary}

We now discuss an example satisfying the hypothesis of above result.
\begin{example}\label{vector-valued-Mz}
Let $\clf$ be a Hilbert space. Recall that the $\mathcal F$-valued Hardy space $H^2(\mathbb D,\mathcal F)$ is defined as
\begin{align*}
    \bigg\{ f(z)=\sum_{n=0}^{\infty}a_nz^n : a_n\in\mathcal F, ~\sum_{n=0}^{\infty}\|a_n\|_{\mathcal F}^2<\infty\bigg\}.
\end{align*}

It is well known that the multiplication operator $M_z$ on
$H^2(\mathbb D,\mathcal F)$ is a pure isometry. Moreover, it follows from
\cite[Section 5.3]{DSS1970} that
$M_z^*$ admits a cyclic vector. Therefore, by Corollary \ref{pure_isometry_admissible}, $M_z$ on $H^2(\mathbb D,\mathcal F)$ is similar to an admissible operator whereas $M_z$ is not admissible due to Proposition \ref{NorIso}.
\end{example}
We are now ready to construct an admissible operator $T$ for which neither
$T$ nor $T^*$ is cyclic. To motivate the example, suppose that
$\cln(T)=\clb(\clh)$. Since $\cls(T)$ (see \eqref{stnt}) is convex,
its SOT-closure is also $\clb(\clh)$
(see \cite[Theorem II.2.6]{T1979}). Consequently, for every
nonzero $h\in\clh$,
\[
\overline{\operatorname{span}}
\{T^{*p}T^qh:p,q\in \Z_+\}=\clh.
\]
In particular, if $T$ has an eigenvalue, then $T^*$ is cyclic. Moreover, when $\operatorname{dim}\clh<\infty$, the admissibility of $T$ implies the cyclicity of both $T$ and $T^*$ (see Section 7.2, Corollary, p. 237 and Section 3, p. 247
in \cite{HK1961}). However, admissibility alone does not, in general, ensure the cyclicity of either $T$ or $T^*$.

\begin{example}
 Consider the Hilbert space $\mathcal{H} = \ell^2(\mathbb{Z}) \otimes \mathbb{C}^2$. 
Let $\{e_n\}_{n \in \mathbb{Z}}$ be the canonical orthonormal basis of $\ell^2(\mathbb{Z})$. Consider the bilateral shift $V$ on $\mathcal{H}$ defined by
\begin{align*}
V(e_n \otimes x) = e_{n+1} \otimes x, \qquad n \in \mathbb{Z},\ x \in \mathbb{C}^2.
\end{align*}
Since $V$ is normal, by Proposition \ref{NorIso}, neither $V$ nor $V^*$ is admissible. We now construct a positive invertible operator $A \in \mathcal{B}(\mathcal{H})$ for which
\[
\cls := \overline{\operatorname{span}}^{\mathrm{UWT}} \{V^{*p} A V^q : p, q \in \mathbb{Z}_+\} = \mathcal{B}(\mathcal{H}).
\]
To verify that $\cls = \mathcal{B}(\mathcal{H})$, it suffices to show that
\begin{align} \label{basisele}
    (e_m \otimes e_n) \otimes P_j \in \cls \qquad (m,n \in \Z,\; j=1,2,3,4),
\end{align}
where $P_1, P_2, P_3, P_4 \in \mathcal{B}(\mathbb{C}^2)$ are positive operators such that
$\{P_1, P_2, P_3, P_4\}$ spans $ \mathcal{B}(\mathbb{C}^2)$. Let $(a_k)_{k \in \N}$ and $(c_k)_{k \in \N}$ be two sequences of natural numbers such that $c_k\in\{1,2,3,4\}$ and no two terms of $a_k$ are equal.
Define the operator $A$ on $\mathcal{H}$ by
\begin{align*}
A = I_{\mathcal{H}} + \sum_{k=1}^{\infty} (e_{a_k} \otimes e_{a_k}) \otimes P_{c_k}.
\end{align*}
Moreover, since each $P_{c_k}$ is positive, we have $I_{\mathcal{H}} \le A \le (1 + M)I_{\mathcal{H}},$ where $M = \sup_{1 \le j \le 4} \|P_j\| < \infty$. Therefore $A$ is positive and invertible.
For each $n \in \mathbb{Z}$, define $X_n \in \mathcal{B}(\mathbb{C}^2)$ by
\begin{align*}
X_n = \begin{cases}
I_{\mathbb{C}^2} + P_{c_k}, & \text{if } n = a_k \text{ for some } k \ge 1, \\
I_{\mathbb{C}^2}, & \text{otherwise}.
\end{cases}
\end{align*}
Then 
\begin{align*}
A(e_n \otimes x) = e_n \otimes X_n x, \qquad n \in \mathbb{Z},\ x \in \mathbb{C}^2.
\end{align*}
Fix $r, d \in \mathbb{Z}$ and $j \in \{1, 2, 3, 4\}$. For any integer $q \ge \max\{0, d\}$, put $p = q - d \ge 0$. Then for every $n \in \mathbb{Z}$ and $x \in \mathbb{C}^2$,
\begin{equation}\label{eq:shift-action}
V^{*p} A V^q (e_n \otimes x)
= V^{*(q-d)} A (e_{n+q} \otimes x)
= V^{*(q-d)} (e_{n+q} \otimes X_{n+q} x)
= e_{n+d} \otimes X_{n+q} x.
\end{equation}
To arrive at \eqref{basisele}, set $a_k = k!$ for $k \ge 1$, and let $(c_k)_{k \in \N}$ be a sequence such that $c_k=k \operatorname{mod}4+1$ for $k \in \N$. Since $a_{k+1}-a_k \rightarrow\infty$ and $j$ occurs infinitely often in $(c_k)_{k \in \N}$, there exist sequences $(q_k)_{k\in \N}$ and $(q_l')_{l\in \N}$ such that 

\begin{equation*}
V^{*(q_k - d)} A V^{q_k} \longrightarrow V^d \qquad \text{in SOT as } k \to \infty,
\end{equation*}
and
\begin{equation*}
V^{*(q_\ell' - d)} A V^{q_\ell'} \longrightarrow V^d + (e_{r+d} \otimes e_r) \otimes P_j \qquad \text{in SOT as } \ell \to \infty.
\end{equation*}
Thus by Lemma \ref{SOT:WOT_UWT}(2), we have
\begin{align*}
(e_{r+d} \otimes e_r) \otimes P_j \in \mathcal{S}.
\end{align*} 
Since $r,d\in\mathbb{Z}$ and $j\in\{1,2,3,4\}$
are arbitrary, we obtain \eqref{basisele}. Thus, by Proposition \ref{similar_admissible}, $A^{1/2}VA^{-1/2}$ is admissible. Since $V$ is a bilateral shift of multiplicity $2$, neither $A^{1/2} V A^{-1/2}$ nor $A^{-1/2} V^* A^{1/2}$ has a cyclic vector (see \cite{C1990}).
\end{example}
\subsection{Admissibility and defect rank} 
For a contraction $T\in\clb(\clh)$, the positive operator $D_T=(I-T^*T)^{1/2}$ is called the defect operator of $T$, and $\operatorname{rank}(I-T^*T)$ is its defect rank. In Example \ref{M_z*-admissible}, observe that
$M_z^*$ is an admissible operator which is co-isometry with defect rank one. Motivated by this, we characterize the admissibility of co-isometries, thereby obtaining a broader class of admissible operators. Recall that $V\in\clb(\clh)$ is a co-isometry if $V^*$ is an isometry.
\begin{proposition} \label{co-isoadm}
    Let $V\in\clb(\clh)$ be a co-isometry. Then $V$ is admissible if and only if one of the following statements hold:
    \begin{enumerate} \item $\clh$ is one-dimensional.
    \item $V$ is completely non-unitary, that is, there exists no nonzero reducing subspace $\mathcal M \subseteq \mathcal H$ of $T$ with $T|_{\clm}$ is unitary, and $\operatorname{rank}(I-V^*V)=1$.
    \end{enumerate} 
\end{proposition}
\begin{proof}
    The forward direction follows using the Wold decomposition of $V^*$ and the argument of the proof of item (2) of Proposition \ref{NorIso}.
        For the converse part, if $\clh$ is one-dimensional, then trivially $V$ is admissible. On the other hand, if $V$ is completely non-unitary and  $\operatorname{rank}(I-V^*V)=1$, $V$ is unitarily equivalent to $M_z^*$ on $H^2(\D)$. Therefore, using Lemma \ref{admtwoproperties} and Example \ref{M_z*-admissible}, $V$ is admissible.
\end{proof}
Considering the operator $M_z$ on $H^2(\D,\clf)$ as in Example \ref{vector-valued-Mz}, if $\dim\mathcal F\geq 2$, by the above proposition $M_z^*$ is not admissible. Due to this and the discussion in Examples \ref{M_z*-admissible} and \ref{vector-valued-Mz}, we record the following.
\begin{remark}\label{M_z}
    Neither $M_z$ nor $M_z^*$ on $H^2(\D,\clf)$ is admissible if $\dim\mathcal F\geq 2$. However $M_z^*$ is admissible when $\dim\mathcal F=1$.  
\end{remark}
We provide another sufficient condition for an operator to be admissible and for this purpose, recall that, an operator $T\in \clb(\clh)$ is said to be pure contraction if $\|T\|\leq 1$ and $T^{* n}\to 0$ as $n\to \infty$ in strong operator topology.
\begin{proposition}\label{pure_rank 1}
Let $T\in \clb(\clh)$ be such that $T^*$ is a pure contraction and $\operatorname{rank}(I-T^*T)=1$. Then $T$ is admissible.
\end{proposition}
\begin{proof}
Since $T^*$ is a pure contraction, $\|T\|=\|T^*\|\leq 1$ and $T^n\to 0$ in SOT as $n\to\infty$. Set $D_T:=(I-T^*T)^{1/2}$ and let $\cld_T:=\overline{\operatorname{ran}}\,D_T$ be the associated defect space. Since $\operatorname{rank}(I-T^*T)=1$, we have $\clh\neq\{0\}$ and $\dim\cld_T=1$. By the Sz.-Nagy--Foia\c{s} model theory \cite{NF}, there exist a nonzero closed $M_z^*$-invariant subspace $\clq\subseteq H^2(\D,\cld_T)$ and a unitary operator $U:\clh\to\clq$ such that
\[
UT^*U^*=P_{\clq}M_z|_{\clq}, \qquad \text{and hence} \qquad UTU^*=M_z^*|_{\clq}.
\]
Because $\dim\cld_T=1$, Remark~\ref{M_z} shows that $M_z^*$ is admissible on $H^2(\D,\cld_T)$. Applying (1) and (2) of Lemma~\ref{admtwoproperties}, it follows that $M_z^*|_{\clq}$ is admissible, and therefore $T$ is admissible.
\end{proof}
\section{Proof of the main theorem}\label{proof_MT}
In this section, we establish Theorem \ref{mainthm}, which answers Question \ref{question1} under suitable
hypotheses. We also investigate a necessary condition for the desired conclusion. We begin with a standard result and include its proof for completeness.

\begin{lemma}\label{commutator}
Let $T\in \clb(\clh)$ be an irreducible operator. Then $\{T, T^*\}^\prime=\C I_\clh$.

\end{lemma}

\begin{proof}
Let $C\in \{T, T^*\}^\prime$, that is,
\[C T =T C\quad \text{and}\quad C T^*=T^* C.\]
Consider 
\[A=\frac{C+C^*}{2} \quad \text{and}\quad B=\frac{C-C^*}{2i}.\]
Then $A$ and $B$ are self-adjoint operators belonging to $\{T, T^*\}'$. By the spectral theorem, every spectral projection of $A$ also commutes with $T$, and its range is therefore a reducing subspace for $T$. Since $T$ is irreducible, each spectral projection of $A$ must be either $0$ or $I_{\clh}$, which implies that $A$ is a scalar multiple of $I_{\clh}$. By the same argument, $B$ is also a scalar multiple of $I_{\clh}$, and hence so is $C = A + iB$. This completes the proof.
\end{proof}

\begin{lemma}\label{lem:ampliation}
Let $\clh$ and $\clk$ be Hilbert spaces. Then the map
\[
\iota:\clb(\clh)\longrightarrow\clb(\clh\otimes\clk),\qquad A\longmapsto A\otimes I_{\clk},
\]
is continuous in UWT.
\end{lemma}
\begin{proof}
Let $\{e_j\}_{j\geq 1}$ be an orthonormal basis of $\clk$. Consider square-summable sequences $(\xi_m)_{m\geq 1}, (\eta_m)_{m\geq 1}$ in $\clh\otimes\clk$, where
\[
\xi_m=\sum_{j=1}^{\infty}x_{m,j}\otimes e_j,\qquad
\eta_m=\sum_{j=1}^{\infty}y_{m,j}\otimes e_j, \quad (x_{m,j}, y_{m,j}\in \clh).
\]
Also, observe that 
\[
\sum_{m,j=1}^{\infty}\|x_{m,j}\|^2=\sum_{m=1}^\infty\|\xi_m\|^2<\infty,
\quad
\sum_{m,j=1}^{\infty}\|y_{m,j}\|^2=\sum_{m=1}^\infty\|\eta_m\|^2<\infty.
\]
 Let $A \in \clb(\clh)$ and a net $\{A_\lambda\} \in \clb(\clh)$ converges to $A$ in UWT. Now
\begin{align*}
\sum_{m=1}^\infty\la (A_\lambda\otimes I_\clk)\xi_m, \eta_m\ra
=\sum_{m,j=1}^{\infty}\la A_\lambda x_{m,j}, y_{m,j}\ra
\rightarrow\sum_{m,j=1}^{\infty}\la Ax_{m,j}, y_{m,j}\ra
=\sum_{m=1}^\infty\la (A\otimes I_\clk)\xi_m, \eta_m\ra.
\end{align*}
Thus, by \eqref{def_uw}, $A_\lambda \otimes I_\clk \to A\otimes I_\clk$ in UWT in $\clb(\clh\otimes \clk)$. This completes the proof.
\end{proof}
The following result plays a crucial role in the proof of main result.
\begin{proposition} \label{ultraweak:reducing}
Let $T_i, \clh_i,\fH,$ and $ \tilde T_i,$ $i=1, \ldots, n$ be as in setup \ref{setup}, and let $\Lambda \subseteq\{1, \ldots, n\} $. Assume that $T_i$ is admissible for each $i\in\Lambda$. If $\fM \subseteq\mathfrak H$ is a jointly $(\tilde T_1,\ldots,\tilde T_n)$-invariant closed subspace, then
\begin{equation*} 
\fM_\Lambda:=
\overline{\operatorname{span}}\bigg\{
\bigg(\prod_{i\in\Lambda}\tilde T_i^{*\alpha_i}\bigg)h:
 h\in\fM,\ \alpha_i\in\mathbb Z_+, i \in \Lambda
\bigg\}
\end{equation*}
 reduces $\tilde T_i$ for each $i\in\Lambda$.
\end{proposition}
\begin{proof}
If $\Lambda$ is empty, then there is nothing to prove. So assume that $\Lambda$ is nonempty. It is easy to see that $\tilde T_i^*(\fM_\Lambda)\subseteq \fM_\Lambda$ for all $i\in \Lambda.$ To complete the proof, it suffices to show that $\tilde T_i(\fM_\Lambda)\subseteq \fM_\Lambda$ for all $i\in \Lambda$.
Now we fix $i\in \Lambda$ and consider a generating vector
\[
x= X\tilde T_i^{* \alpha_i} h \in \fM_\Lambda, \quad\text{where}\,\, 
X=\prod_{j\in \Lambda\setminus\{i\}}\tilde T_j^{* \alpha_j} \mbox{ and } h\in\fM,
\]
(note that $X=I$ if $\Lambda$ is singleton). Since $T_i$ is admissible, there exists a net $\{F_\lambda\}\in\cls(T_i)$ (see \eqref{stnt}) such that
\begin{align*} 
 F_\lambda=\sum_{p,r}c_{p,r}^{(\lambda)}T_i^{* p}T_i^{r} \to T_i T_i^{* \alpha_i}\quad \text{in UWT},   
\end{align*}
where $c_{p,r}^{(\lambda)}\in \C$.
By Lemma \ref{lem:ampliation}, 
\[
\widetilde F_\lambda:=I_{\mathcal{H}_1}\otimes\cdots\otimes I_{\mathcal{H}_{i-1}}\otimes F_\lambda \otimes I_{\mathcal{H}_{i+1}}\otimes\cdots\otimes I_{\mathcal{H}_n}\longrightarrow\widetilde T_i\widetilde T_i^{*\alpha_i}
\quad\text{ in UWT}.
\]
 Thus, we have $X \widetilde F_\lambda h \rightarrow X (\tilde T_i \tilde T_i^{*\alpha_i})h$ in $\fH$. This together with $X \widetilde F_\lambda h\in \fM_\Lambda$ shows that 
\[\tilde T_i x =X(\tilde T_i \tilde T_i^{*\alpha_i})h\in \fM_\Lambda,\]
since $(\tilde T_1, \ldots, \tilde T_n)$ is doubly commuting.
The rest of the proof follows as $\tilde T_i$ is bounded, for all $i=1,\ldots,n$ and $h \in \fM$ are arbitrary.
\end{proof}
Let $T_i, \clh_i,\fH,$ and $ \tilde T_i,$ $i=1, \ldots, n$ be as in Setup \ref{setup}. Given an auxiliary Hilbert space $\clk$, define
 \[Y_i:\clk \otimes \fH \to \clk \otimes \fH \quad \text{by}\quad Y_i:=I_{\clk} \otimes \tilde T_i \quad (i=1,\ldots, n).\]
 Then it is easy to see that $(Y_1, \ldots, Y_n)$ on $\clk \otimes \fH$ is doubly commuting. Under the assumption that each $T_i$ is irreducible on $\clh_i$ for $i=1, \ldots, n$, we obtain the following characterization of jointly $(Y_1,\ldots,Y_n)$-reducing subspaces of $\clk\otimes\fH$.

\begin{proposition} \label{reducingirr}
 Let $T_i$ be an irreducible operator on $\clh_i$, for all $i=1,\ldots,n$. Then, with the notation as above, a closed subspace $\fM$ of $\clk \otimes \fH$ is jointly reducing under $(Y_1,\ldots,Y_n)$ if and only if there exist a closed subspace $\mathcal L$ of $\mathcal K$ such that $$\fM=\mathcal L \otimes \fH.$$
\end{proposition}
\begin{proof}
Let $\fM$ be a jointly reducing subspace of $\clk \otimes \fH$ under $(Y_1,\ldots,Y_n)$. Then $$P_\fM Y_j=Y_jP_\fM, $$
for all $j=1, \ldots, n$. Since $T_i$ is irreducible for each $i=1, \ldots, n$,  it follows from Lemma \ref{commutator} and double commutant Theorem (\cite[Theorem II.3.9]{T1979}) that von Neumann algebra generated by $\{I_{\clh_i}, T_i\}$ is equal to $\mathcal B(\clh_i)$ for each $i=1, \ldots, n.$ Thus $P_\fM \in (I_{\clk} \otimes \mathcal B(\fH))'.$ It now follows from \cite[Theorem IV.5.9]{T1979} that $$P_\fM \in \mathcal B(\mathcal K) \otimes I_{ \fH}.$$ Thus there exists an operator $X \in \clb(\clk)$ such that $P_\fM=X \otimes I_{\fH}.$ Since $P_\fM$ is an orthogonal projection, $X$ is also an orthogonal projection. Thus $\fM=\mathcal L \otimes \fH,$ where $\mathcal L=\operatorname{ran}X\subseteq \clk.$ The converse part is trivial.
\end{proof}
We now have all the necessary ingredients to prove our main result.
\begin{proof}[\textbf{Proof of Theorem \ref{mainthm}}]
Suppose $\fM$ is a doubly commuting jointly $(\tilde T_1, \ldots, \tilde T_n)$-invariant closed subspace of $\fH$, that is,
\[
(\tilde T_i|_{\fM})(\tilde T_j|_{\fM})^* = (\tilde T_j|_{\fM})^*(\tilde T_i|_{\fM})
\]
for $1 \leq i < j \leq n$ (see Setup \ref{setup}). If $\fM = \{0\}$, the result holds trivially by taking $\clm_i = \{0\}$ for all $i = 1, \ldots, n$. Assume therefore that $\fM \neq \{0\}$. Define
\begin{align} \label{m_f}
\fM_1 := \overline{\operatorname{span}}\{\tilde T_2^{*k_2}\cdots \tilde T_n^{*k_n}\fM : k_2, \ldots, k_n \in \mathbb Z_+\}.
\end{align}
Since $\fM \subseteq \fM_1$, $\fM_1 \neq \{0\}$. Since $T_i$, for $i = 2, \ldots, n$, are admissible, an application of Proposition \ref{ultraweak:reducing} implies that $\fM_1$ reduces $(\tilde T_2, \ldots, \tilde T_n)$, and by Proposition \ref{admtostrng} each $T_i$ is irreducible. We now apply Proposition \ref{reducingirr} to $\clh_1 \otimes (\clh_2 \otimes \cdots \otimes \clh_n)$ so that there exists a closed subspace $\clm_1 \subseteq \clh_1$ such that
\[
\fM_1 = \clm_1 \otimes \clh_2 \otimes \cdots \otimes \clh_n.
\]
Note that $\clm_1 \neq \{0\}$ and $\clh_2 \otimes \cdots \otimes \clh_n \neq \{0\}$ since $\fM_1 \neq \{0\}$. Because $\tilde T_1$ commutes with $\tilde T_j^*$ for all $j = 2, \ldots, n$ and leaves $\fM$ invariant, we have $\tilde T_1(\tilde T_2^{*k_2}\cdots \tilde T_n^{*k_n}\fM) = \tilde T_2^{*k_2}\cdots \tilde T_n^{*k_n}\tilde T_1\fM \subseteq \fM_1$, whence $\fM_1$ is $\tilde T_1$-invariant. Fixing a nonzero vector $\eta \in \clh_2 \otimes \cdots \otimes \clh_n$, for any $x \in \clm_1$ we obtain $(T_1x) \otimes \eta = \tilde T_1(x \otimes \eta) \in \clm_1 \otimes \clh_2 \otimes \cdots \otimes \clh_n$, which is orthogonal to $y \otimes \eta$ for every $y \in \clm_1^{\perp}$; hence $T_1x \in \clm_1$. Thus $\clm_1$ is a nonzero closed $T_1$-invariant subspace of $\clh_1$.

We now show that $\fM$ reduces $\tilde T_1|_{\fM_1} = (T_1|_{\clm_1}) \otimes I_{\clh_2} \otimes \cdots \otimes I_{\clh_n}$. Since $\fM \subseteq \fM_1$, 
this is equivalent to showing that
\begin{align} \label{fmreducing}
    P_{\fM}(\tilde T_1|_{\fM_1}) = (\tilde T_1|_{\fM_1})P_{\fM} \qquad \text{on }\,\, \fM_1.
\end{align}
For each $j = 2, \ldots, n$, since $\fM$ is $\tilde T_j$-invariant, we have $(\tilde T_j|_{\fM})^* = P_{\fM}\tilde T_j^*|_{\fM}$ and 
\begin{align}\label{id}
P_{\fM}\tilde T_j^*P_{\fM} = P_{\fM}\tilde T_j^*(I-P_{\fM^\perp})=P_{\fM}\tilde T_j^*-P_{\fM}\tilde T_j^*P_{\fM^\perp}=P_{\fM}\tilde T_j^*,
\end{align}
where the last equality follows as $\fM^\perp$ is $\tilde T_j^*$-invariant. Repeated application of \eqref{id} gives, for any $k_2, \ldots, k_n \in \mathbb Z_+$,
\[
\prod_{j=2}^n (\tilde T_j|_{\fM})^{*k_j} = P_{\fM}\bigg(\prod_{j=2}^n \tilde T_j^{*k_j}\bigg)\bigg|_{\fM}.
\]
Note that $(\tilde T_1|_{\fM})(\tilde T_j|_{\fM})^* = (\tilde T_j|_{\fM})^*(\tilde T_1|_{\fM})$ for each $j = 2, \ldots, n$, and that $\tilde T_1$ commutes with each $\tilde T_j^*$. Therefore,
\begin{align*}
(\tilde T_1|_{\fM}) &P_{\fM}\bigg(\prod_{j=2}^n \tilde T_j^{*k_j}\bigg)\bigg|_{\fM}
= (\tilde T_1|_{\fM}) \prod_{j=2}^n (\tilde T_j|_{\fM})^{*k_j} 
= \bigg(\prod_{j=2}^n (\tilde T_j|_{\fM})^{*k_j}\bigg)(\tilde T_1|_{\fM}) \\
&= P_{\fM}\bigg(\prod_{j=2}^n \tilde T_j^{*k_j}\bigg)(\tilde T_1|_{\fM}) 
= P_{\fM}\tilde T_1\bigg(\prod_{j=2}^n \tilde T_j^{*k_j}\bigg)\bigg|_{\fM}
= P_{\fM}(\tilde T_1|_{\fM_1})\bigg(\prod_{j=2}^n \tilde T_j^{*k_j}\bigg)\bigg|_{\fM}.
\end{align*}
Now, for any $x \in \fM$, 
since $\fM\subseteq\fM_1$, we have
\begin{align*}
P_{\fM}(\tilde T_1|_{\fM_1})\bigg(\prod_{j=2}^n \tilde T_j^{*k_j}\bigg)x
= (\tilde T_1|_{\fM}) P_{\fM}\bigg(\prod_{j=2}^n \tilde T_j^{*k_j}\bigg)x
= (\tilde T_1|_{\fM_1}) P_{\fM}\bigg(\prod_{j=2}^n \tilde T_j^{*k_j}\bigg)x.
\end{align*}
Therefore, by continuity and the definition of $\fM_1$ in \eqref{m_f} yields \eqref{fmreducing} as desired.

Since $T_1$ is admissible, by Proposition \ref{admtostrng}, $T_1|_{\clm_1}$ is irreducible. Thus, by Proposition \ref{reducingirr} (applied to $\clm_1 \otimes (\clh_2 \otimes \cdots \otimes \clh_n)$ after interchanging the two tensor factors), there exists a closed subspace $\mathfrak F$ of $\clh_2 \otimes \cdots \otimes \clh_n$ such that
\[
\fM = \clm_1 \otimes \mathfrak F.
\]
Here we clarify that when acting on $\clh_2 \otimes \cdots \otimes \clh_n$, the tuple $(\tilde T_2, \ldots, \tilde T_n)$ is defined in compliance with Setup \ref{setup} considering the operators $T_i$ on $\clh_i$, for $i = 2, \ldots, n$, so that the $i$-th amplified operator on $\fH$ is $I_{\clh_1} \otimes \tilde T_i$. Fixing a nonzero vector $x_0 \in \clm_1$, for any $\eta \in \mathfrak F$ and $i = 2, \ldots, n$, we have $x_0 \otimes \tilde T_i\eta = (I_{\clh_1} \otimes \tilde T_i)(x_0 \otimes \eta) \in \fM = \clm_1 \otimes \mathfrak F$, which implies $\tilde T_i\eta \in \mathfrak F$. Hence $\mathfrak F$ is jointly $(\tilde T_2, \ldots, \tilde T_n)$-invariant. Note also that for $2 \leq i < j \leq n$, doubly commutativity of $\fM$ yields
\[
(I_{\clh_1} \otimes \tilde T_i)P_{\fM}(I_{\clh_1} \otimes \tilde T_j^*)|_{\fM}
= P_{\fM}(I_{\clh_1} \otimes \tilde T_j^*)|_{\fM}(I_{\clh_1} \otimes \tilde T_i)|_{\fM}.
\]
Since $P_{\fM} = P_{\clm_1} \otimes P_{\mathfrak F}$, this identity becomes
\[
I_{\clm_1} \otimes (\tilde T_i P_{\mathfrak F}\tilde T_j^*|_{\mathfrak F})
= I_{\clm_1} \otimes (P_{\mathfrak F}\tilde T_j^*|_{\mathfrak F}\tilde T_i|_{\mathfrak F})
\quad \text{on } \,\clm_1 \otimes \mathfrak F.
\]
Because $\clm_1 \neq \{0\}$, we obtain for each $2 \leq i < j \leq n$,
\begin{align*}
\tilde T_i P_{\mathfrak F}\tilde T_j^*|_{\mathfrak F} = P_{\mathfrak F}\tilde T_j^*|_{\mathfrak F}\tilde T_i|_{\mathfrak F}.
\end{align*}
This shows that $\mathfrak F$ is a doubly commuting jointly $(\tilde T_2, \ldots, \tilde T_n)$-invariant closed subspace of $\clh_2 \otimes \cdots \otimes \clh_n$. Now continuing in a similar manner (by induction on $n$), there exist closed subspaces $\clm_i \subseteq \clh_i$, each invariant under the corresponding operator $T_i$, for $i = 2, \ldots, n$, such that
\[
\fM = \clm_1 \otimes \cdots \otimes \clm_n.
\]
The converse part is trivial.
\end{proof}
Upon examining the proof of Theorem \ref{mainthm}, one observes that it is not necessary for all $n$ operators to be admissible: the conclusion remains valid if one of the operators is merely assumed to be fully irreducible, a strictly weaker condition than admissibility, while the remaining $n-1$ operators are admissible. Thus, we have the following result.

\begin{theorem} \label{weakmainthm}
Under Setup \ref{setup}, suppose that at least $n-1$ of the operators $T_1,\ldots,T_n$ are admissible and the remaining operator is fully irreducible. Let $\fM$ be a closed jointly invariant subspace of $\fH$ under $(\tilde T_1, \ldots, \tilde T_n)$. Then $\fM$ is doubly commuting if and only if $\fM$ is hereditary.
\end{theorem}
We now discuss an immediate application of Theorem \ref{mainthm}. Recall that $M_z^*$ on $H^2(\D)$ is an admissible operator (see Example \ref{M_z*-admissible}). With the notation of Setup \ref{setup}, setting $\clh_i=H^2(\D)$ and $T_i=M_z^*$ for $i=1,\ldots,n$, we recover the following result of \cite[Theorem 3.2]{S} as a consequence of Theorem \ref{mainthm}.
\begin{corollary} \label{hardyresultrecovery}
    Let $\fQ$ be a quotient module of $H^2(\D^n)$. Then $\fQ$ is doubly commuting if and only if there exist quotient modules $\clq_1, \ldots, \clq_n$ of $H^2(\D)$ such that 
    \[\fQ=\clq_1\otimes\cdots\otimes\clq_n.\]
\end{corollary}

We now show that fully irreducibility is a necessary condition to have an affirmative answer to Question \ref{question1}.

\begin{proposition} \label{dctosi}
Let $n \ge 2$, and let $\clh_i, \fH, T_i, \tilde{T}_i$ ($i=1,\ldots,n$) be as in Setup \ref{setup}. Suppose that every doubly commuting jointly $(\tilde{T}_1, \ldots, \tilde{T}_n)$-invariant closed subspace $\fM$ of $\fH$ is hereditary. Then the following statements are true.
\begin{enumerate}
    \item If there are at least two indices $r,s \in \{1,\ldots,n\}$, such that $T_r$ and $T_s$ has a nonzero proper closed invariant subspace. Then $T_i$ is fully irreducible for every $i = 1, \ldots, n$.
    \item If $T_i$ is not fully irreducible, then each $T_j$ has only $\{0\}$ and $\clh$ as invariant subspaces for $j\neq i.$
    \item At least $n-1$ $T_i$'s are fully irreducible.
\end{enumerate}
\end{proposition}
\begin{proof}
(1) Suppose $T_1$ is not fully irreducible. There exists a nonzero closed $T_1$-invariant space $\clw \subseteq\clh_1$ (possibly $\clw=\clh_1$) and an orthogonal decomposition $\clw=\clw_1\oplus \clw_2$, where $\clw_1$ and $\clw_2$ are nonzero and reduce $T_1|_\clw$. By the hypothesis, there exists $j\neq 1$ such that $T_j$ has a nonzero proper closed invariant subspace $\cla_j\subsetneq\clh_j$. Write
\[
\mathcal K= \clh_2 \otimes \cdots \otimes\clh_n,
\qquad \mathcal L= \left(\bigotimes_{\ell\neq 1,j}\clh_\ell\right)\otimes \cla_j,
\]
with the factors arranged in the original order, so $\{0\}\subsetneq\mathcal L\subsetneq\mathcal K$. Set
\begin{align*} 
   \mathcal Q=(\clw_1\otimes\mathcal K)\oplus(\clw_2\otimes\mathcal L).
\end{align*}

It is easy to see that $\mathcal Q$ is doubly commuting jointly $(\tilde{T}_1, \ldots, \tilde{T}_n)$-invariant and it cannot be written as a single tensor product. Thus $T_1$ is fully irreducible. Using the similar technique, it follows that each $T_i$ is fully irreducible for $i=2, \ldots, n.$

\noindent 
(2) It follows immediately from (1) through the contradiction argument. 

\noindent 
(3) If all the $T_i$ are fully irreducible, there is nothing to prove. Otherwise fix an index $i$ such that $T_i$ is not fully irreducible. By (2), every $T_j$ with $j\neq i$ has only the two closed invariant subspaces, namely, $\{0\}$ and $\clh$. It is easy to see that any such operator is fully irreducible. Thus at least $n-1$ of the operators are fully irreducible.
\end{proof}

Although Proposition \ref{dctosi} suggests that the admissibility assumption in Theorem \ref{mainthm}
might be weakened to full irreducibility of at least $n-1$ of the operators $T_1,\ldots,T_n$, this is not the case. To see this, take the shift operator $M_z$ on $H^2(\D)$. The operator $M_z$ is fully irreducible. Doubly commuting jointly $(M_{z_1},\ldots,M_{z_n})$-invariant subspaces of $H^2(\D^n)$ may not be hereditary (see \eqref{mzsetup}). For example, consider $n=2$ and $q(z_1,z_2)=\frac{z_1z_2-a}{1-az_1z_2}$ for $0<a<1.$ Since $q$ is inner, $qH^2(\D^2)$ is doubly commuting \cite{M} but not hereditary as $q$ is not variable separating (see \cite[p. 3225]{GW}).

\section{Applications to de Branges-Rovnyak spaces}\label{Hb}
Let $b\in H^{\infty}_1(\D)$, the closed unit ball of the Banach algebra $H^{\infty}(\D)$ of bounded analytic functions on $\D$, equipped with the supremum norm. The associated de Branges-Rovnyak space, denoted by $H(b)$ (see \cite{FM2} for an exposition), is a reproducing kernel Hilbert space with kernel given by
\begin{align} \label{debrangeskernel}
  K^b(z,w):=k^b_w(z)=(1-b(z)\overline{b(w)})\mathbb S_w(z), \mbox{ where } \mathbb S_w(z)=\frac{1}{1-z\overline w} \quad (z,w\,\in\D).  
\end{align}
The space $H(b)$ is a linear subspace of $H^2(\D)$, which need not be closed in the $H^2(\D)$-norm. It follows from \cite[Theorem 18.13]{FM2} that $H(b)$ is invariant under the backward shift operator $M_z^*$ on $H^2(\D)$. So, \[X_b:H(b) \to H(b), \quad \text{defined by }\quad X_b:=M^*_z|_{H(b)}\] 
is a well-defined operator. In this section, we show that for a nonconstant extreme point $b$, both $X_b$ and $X_b^*$ are admissible. For the nonextreme points, we characterize $b$ for which $X_b^*$ is admissible, and we give a condition under which $X_b$ is admissible (see \cite[Chapter 6]{FM1} for an exposition to extreme and nonextreme points). Then we characterize doubly commuting jointly invariant subspaces of the tensor product of $H(b)$ spaces under the tuples associated to $X_b$ and $X_b^*$ as in Setup \ref{setup}. We begin with the following elementary result.
\begin{lemma}\label{k_0-cyclic vector}
    Let $b\in H^{\infty}_1(\D)$. Then \[
\overline{\operatorname{span}}
\{X_b^{*n}k_0^b:n\in \Z_+\}
= H(b).
\]
\end{lemma}
\begin{proof}
Let $f(z)=\sum_{j=0}^{\infty}a_jz^j \perp _b\overline{\operatorname{span}}
\{X_b^{*n}k_0^b:n\in \Z_+\}.$ Then, since $k_0^b$ is the reproducing kernel of $ H(b)$ at the
origin (see \eqref{debrangeskernel}), 
\begin{align} \label{coeffi-zero}
(X_b^nf)(0)=\langle X_b^nf,k_0^b\rangle_{ H(b)}=\langle f,X_b^{*n}k_0^b\rangle_{H(b)}=0,   
\end{align}
for all $n\in\Z_+$. On the other hand, it follows that 
$X_b^nf(z)
=
\sum_{j=0}^{\infty}a_{j+n}z^j,$ and consequently $(X_b^nf)(0)=a_n.$ This combined with \eqref{coeffi-zero} yields that $a_n=0$ for every $n\in\Z_+$, and therefore $f=0$. This completes the proof.
\end{proof}
\subsection{Tensor products of \texorpdfstring{$H(b)$}{H(b)} spaces for extreme points \texorpdfstring{$b$}{b} and Jordan blocks} \label{Hb_extreme}
Let $b\in H^\infty_1(\D)$ be a nonconstant extreme point.
 It follows from \cite[Corollary 18.23 and Theorem 25.11]{FM2} that 
 \begin{align}\label{bothdefect1}
  I-X^*_b X_b=k^b_0\otimes k_0^b  \quad \text{and}\quad I-X_bX^*_b= M^*_z b \otimes M^*_z b,    \end{align}
where $k^b_0=1-\overline{b(0)}b$ (see \eqref{debrangeskernel}). Also, from Lemma \ref{k_0-cyclic vector}, 
\[\overline{\text{span}}\{X_b^{*k}k^b_0: k\in \Z_+ \}= H(b).\]
Therefore, by Proposition \ref{cyclic_vector}, $X_b$ is an admissible operator.
On the other hand, by \cite[Corollary 26.18]{FM2}
\[
\overline{\operatorname{span}}
\{X_b^{n}(M^*_z b):n\in \Z_+\}=\overline{\operatorname{span}}
\{M_z^{* n+1}(b):n\in \Z_+\}=H(b).
\]
This combined with 
\eqref{bothdefect1} gives that $X^*_b$ is also an admissible operator on $H(b)$ (see Proposition \ref{cyclic_vector}). The preceding discussion yields the following.
\begin{proposition} \label{extxbxbstar}
    If $b \in H^\infty_1(\D)$ is nonconstant extreme point, then both $X_b$ and $X^*_b$ are admissible operators.
\end{proposition}
Combining the above result with Theorem \ref{mainthm}, we obtain the following result, which not only characterize the doubly commuting jointly invariant and co-invariant subspaces of tensor product of $H(b)$ spaces, but also provide a characterization of doubly commuting jointly mixed invariant subspaces in this setting.
\begin{theorem} \label{mixed}
In Setup~\ref{setup}, let $\clh_i = H(b_i)$ for nonconstant extreme points $b_1,\ldots,b_n \in H_1^\infty(\D)$, and for each $i=1,\ldots,n$, let $T_i$ be either $X_{b_i}$ or $X_{b_i}^*$. Then a closed jointly $(\tilde T_1,\ldots,\tilde T_n)$-invariant subspace $\fM \subseteq \fH$ is doubly commuting if and only if it is hereditary.
\end{theorem}
In particular, if $b=\theta$ is a nonconstant inner function, then $H(b)$ coincides with the quotient module $\clq_\theta=H^2(\D)\ominus \theta H^2(\D)$ of $H^2(\D)$ (see \cite[Theorem 18.1]{FM2}). Consider the operator $X_{\theta}=M_z^*|_{\clq_\theta}$ (this is nothing but $S_\theta^*$ following the notation of \cite{SSS}). Therefore, considering nonconstant inner functions $\theta_1,\ldots,\theta_n$ and the Jordan block   $\clq_{\theta_1}\otimes\cdots\otimes\clq_{\theta_n}$,
taking $\Lambda=\emptyset$ and $T_i=X_{\theta_i}^*=S_{\theta_i}$ in Theorem \ref{mixed}, as a corollary we recover the following result of \cite[Theorem 3.2]{S}.
\begin{corollary}\label{SSS_Theorem}
    Let $\theta_1,\ldots,\theta_n$ be nonconstant inner functions and let $\fM$ be a jointly invariant subspace of the Jordan block $\clq_{\theta_1}\otimes\cdots\otimes\clq_{\theta_n}$ under $(S_{\theta_1}, \ldots, S_{\theta_n})$. Then $\fM$ is doubly commuting if and only if there exist $S_{\theta_i}$-invariant subspaces $\clm_i$ of $\clq_{\theta_i}$, for each $i=1,\ldots, n$, such that
    \[\fM=\clm_1\otimes\cdots\otimes\clm_n.\]
\end{corollary}
We now turn our attention to tensor product of mixed spaces. For $m,n\in\N$ and given the nonconstant extreme points $b_1,\ldots,b_n\in H_1^{\infty}(\D)$, define the mixed space 
\[
\clh^2(b):=H^2(\D^m)\otimes H(b_1)\otimes\cdots\otimes H(b_n).
\]
Also define a operator tuple $(X_1,\ldots,X_{m}, Y_1,\ldots,Y_n)$ on $\clh^2(b)$ by
\[
X_i:=M_{z_i}\otimes I_{H(b_1)}\otimes\cdots\otimes I_{H(b_n)},
\]
for all $i=1,\ldots,m$ and 
\[
Y_i:= I_{H^2(\D^m)}\otimes \tilde T_i,
\]
where $T_i$ is either $X_b$ or $X_b^*$ for all $i=1,\ldots,n$. It is worth mentioning that Theorem \ref{mixed} along with \cite[Proposition 6.1]{SSS} provides the following generalization of \cite[Theorem 6.2]{SSS}.
\begin{theorem} \label{mixedhd}
    Using the notation defined above, a closed subspace $\fM$ of $\clh^2(b)$ is doubly commuting and jointly $(X_1,\ldots,X_m,Y_1,\ldots,Y_n)$-invariant if and only if
    \[
    \fM = \varphi H^2(\D^m) \otimes \clm_1 \otimes \cdots \otimes \clm_n,
    \]
    for some inner function $\varphi \in H^\infty(\D^m)$ and closed $T_i$-invariant subspaces $\clm_i \subseteq H(b_i)$ for $i=1,\ldots,n$.
\end{theorem}
\subsection{Tensor products of \texorpdfstring{$H(b)$}{H(b)} spaces for nonextreme points \texorpdfstring{$b$}{b}} \label{Hb_nonextreme}
Let $b$ be a nonextreme point of $H_1^\infty(\mathbb D)$, and let $a$ denote the associated outer function (see \cite[Section 23.1]{FM2}), normalized so that $a(0)>0$ and satisfying
\[
|a|^2+|b|^2=1 \quad \text{a.e. on } \mathbb T,
\]
where $\mathbb T$ is the unit circle in $\mathbb C$. The function $a$ is also called as the Pythagorean mate associated with $b$. Using the notation that $T_\psi$ denotes the Toeplitz operator with respect to symbol $\psi \in L^\infty(\mathbb T)$ on $H^2(\D)$, for nonextreme $b$, we recall from \cite[Section 23.3]{FM2} that
\begin{align} \label{ffplusrela}
 f\in H(b) \text { if and only if there is } f^+\in H^2(\D) \text{ such that
} T_{\bar b}f=T_{\bar a}f^+. 
\end{align}
Then the following relation holds:
\begin{align} \label{inprodhb}
  \langle f, g \rangle_{H(b)}=\langle f, g \rangle_{H^2(\D)}+\langle f^+, g^+ \rangle_{H^2(\D)}, \qquad (f,g \in H(b)).
\end{align}
We first completely characterize the nonextreme points $b$ for which $X_b^*$ is admissible.

\begin{theorem} \label{Xbstaradm}
    Let $b$ be a nonextreme point in $H^\infty_1(\D).$ Then the following statements are equivalent:
    \begin{enumerate}
        \item $X_b^*$ is admissible on $H(b)$.
        \item $M_z^*b$ is cyclic for $X_b$ on $H(b)$.
        \item $M_z^*b$ is cyclic for $M_z^*$ on $H^2(\D)$.
        \item $b$ is cyclic for $M_z^*$ on $H^2(\D)$.
        \item $b$ has no bounded type meromorphic pseudocontinuation to $\{z: 1 <|z|\leq \infty\}$.
    \end{enumerate}
\end{theorem}
\begin{proof}
If $b$ is constant, then $H(b)=H^2(\D)$, and $X_b^*$ is unitarily
equivalent to $M_z$. Consequently, $X_b^*$ is not admissible (see Example \ref{M_z-not admi} and Lemma \ref{admtwoproperties}). Moreover, $b$ is not cyclic for
$M_z^*$. Thus, all the assertions fail in this case.

Assume now that $b$ is nonconstant, so that $M_z^*b \neq 0$. Recall from
\cite[Corollary 18.23]{FM2} that
\[
I-X_bX_b^*=M_z^*b\otimes M_z^*b.
\]
The equivalence $(1)\iff(2)$ follows from Corollary
\ref{admiffcyclic}, while $(2)\iff(3)$ follows from
\cite[Corollary 24.32]{FM2}. Further, by \cite[Theorem 2.2.8]{DSS1970},
a function $b\in H^2(\D)$ is cyclic for
$M_z^*$ if and only if $zb$ is cyclic
for $M_z^*$. Since
\[
b=b(0)+zM_z^*b,
\]
it follows that $b$ is cyclic for $M_z^*$
if and only if $M_z^*b$ is cyclic for
$M_z^*$. Thus we get $(3)\iff(4)$. Finally, $(4)\iff(5)$ follows from
\cite[Theorem 2.2.1]{DSS1970}.
\end{proof}
Now we turn our attention to the admissibility of $X_b$. Recall from \cite[Theorem 23.14]{FM2} that 
\begin{align} \label{defectD}
    D:=I-X_b^*X_b=k^b_0\otimes k_0^b+a(0)^2 b \otimes b.
\end{align} 
If $b$ is nonconstant, then $k_0^b$ and $b$ are linearly independent. Hence $I-X_b^*X_b$ is of rank $2.$
We begin by identifying a broad class of nonextreme points for which $X_b$ fails to be admissible.
\begin{lemma} \label{non-extre not adm}
    Let $b \in H^\infty_1(\mathbb D)$ and $\varphi:=b/a.$ If $\varphi \in H^\infty(\D)$ is nonconstant and there exist $\alpha$ and $\beta $ in $\mathbb D$ such that $1+\overline{\varphi(\alpha)}\varphi(\beta)=0$ then $X_b$ is not admissible.
\end{lemma}
\begin{proof}
 First note that for any $f \in H^2(\D),$ we have 
 $$T_{\bar b}f =T_{\bar a}T_{\bar \varphi}f.$$ It follows from \eqref{ffplusrela} that $H(b)=H^2(\D)$ and $f^+=T_{\bar \varphi}f.$ Now, by \eqref{inprodhb}, we have $$\langle f, g \rangle_{H(b)}=\langle f, g \rangle_{H^2(\D)}+\langle T_{\bar \varphi}f, T_{\bar \varphi}g \rangle_{H^2(\D)} \qquad (f, g \in H(b)).$$ Consider $f_1(z)=\frac{1}{1-\overline\alpha z}$ and $f_2(z)=\frac{1}{1-\overline \beta z}.$ Then 
  \begin{align*}
        \langle f_1, f_2 \rangle_{H(b)}&=\langle f_1, f_2 \rangle_{H^2(\D)}+\langle T_{\bar \varphi}f_1, T_{\bar \varphi}f_2 \rangle_{H^2(\D)} \\&=\langle f_1, f_2 \rangle_{H^2(\D)}+\overline{\varphi(\alpha)} \varphi(\beta)\langle f_1, f_2 \rangle_{H^2(\D)} \\&=\left(1+\overline{\varphi(\alpha)} \varphi(\beta)\right) \langle f_1, f_2 \rangle.
\end{align*}
It follows from the assumption that $\langle f_1, f_2 \rangle_{H(b)}=0.$ 
Also, note that
\begin{align*}
X_bf_1=\bar \alpha f_1\quad \text{and}\quad X_bf_2=\bar \beta f_2.
\end{align*}
Consider $\mathcal M=\operatorname{span}\{f_1, f_2\}.$ Therefore, $\mathcal M$ is an invariant subspace of $X_b$ and $X_b|_{\clm}$ is reducible. Thus an application of Proposition \ref{admtostrng} yields that $X_b$ is not admissible.
\end{proof}
In the following example, we exhibit a class of nonextreme functions $b$ satisfying the hypotheses of Lemma \ref{non-extre not adm}.
\begin{example} \label{nonextreme}
   Consider $b(z)=\beta B(z)$, where $0<|\beta|<1$ and $B$ is a finite Blaschke product of the form
\[
B(z)=\prod_{i=1}^n c_i\frac{z-\alpha_i}{1-\overline{\alpha_i}z},
\]
for some $c_i\in\mathbb T$ and $\alpha_i\in\D$, $i=1,\ldots,n$. Then $a=\sqrt{1-|\beta|^2}$ and it follows from \eqref{ffplusrela} that $H(b)$ is equal to $H^2(\D)$ as a set. Then $\varphi=\frac{\beta}{\sqrt{1-|\beta|^2}} B$ is bounded. Thus, $1+\overline{\varphi(z_1)}\varphi(z_2)=0$ has a solution $z_1, z_2 \in \D$ if and only if $2|\beta|^2>1.$ Indeed,
   \begin{align*}
      1+\overline{\varphi(z_1)}\varphi(z_2)=0 \iff1+\frac{|\beta|^2}{1-|\beta|^2}\overline{B(z_1)}B(z_2)=0 \iff \overline{B(z_1)}B(z_2)=-\frac{1-|\beta|^2}{|\beta|^2}. 
   \end{align*}
   Since $|\overline{B(z_1)}B(z_2)|<1$ for $z_1, z_2 \in \D,$ we have 
      \begin{align*}
      1+\overline{\varphi(z_1)}\varphi(z_2)=0 \iff\frac{1-|\beta|^2}{|\beta|^2}<1 \iff2|\beta|^2>1.
   \end{align*}
 In this situation, $X_b$ is not admissible.  \end{example}
We now focus on identifying nonextreme $b$ for which $X_b$ is admissible. Proposition \ref{cyclic_vector} and Lemma \ref{k_0-cyclic vector} readily provides an equivalent criteria for $X_b$ to be admissible, which is stated as follows.
\begin{proposition} \label{necesuffi}
For nonextreme $b\in H^{\infty}_1(\D)$, $\cln(X_b)=\clb(H(b))$ if and only if $k^b_0\otimes k_0^b \in \cln(X_b).$
\end{proposition}
Now, it is easy to see that 
\begin{align*}
    \overline{\operatorname{span}}^{\mathrm{UWT}}\{X_b^{*p}DX^q_b : p,q\in \Z_+\}\subseteq \cln(X_b),
\end{align*}
where $D$ is as in \eqref{defectD}. Since $$X_b^{*m}X^n_b=\sum_{k=0}^{r-1}X_b^{*(m+k)}DX^{(n+k)}_b +X_b^{*(m+r)}X^{(n+r)}_b$$ and $X_b^l \rightarrow 0$ as $l \rightarrow \infty$ in SOT, 
$$\sum_{k=0}^{r-1}X_b^{*(m+k)}DX^{(n+k)}_b \rightarrow X_b^{*m}X^n_b$$ in SOT. By \cite[Theorem 18.13]{FM2}, $\|X_b^{*(m+r)}X^{(n+r)}_b\|\leq 1$ for each $m,n,r \in \Z_+.$ By Lemma \ref{SOT:WOT_UWT}(2), we have $X_b^{*m}X_b^n \in \overline{\operatorname{span}}^{\mathrm{UWT}}\{X_b^{*p}DX^q_b : p,q\in \Z_+\}$ for each $m,n \in \mathbb Z_+$. The preceding discussion yields the following result. 
\begin{lemma} \label{nonext: Dmiddle}
For a nonextreme point $b\in H^\infty_1(\D)$, $$\overline{\operatorname{span}}^{\mathrm{UWT}}\{X_b^{*p}DX^q_b : p,q\in \Z_+\}= \cln(X_b).$$

\end{lemma}
In the following, we prove a sufficient condition for the admissibility of the operator $X_b.$
\begin{theorem} \label{non-extre adm}
    Let $b$ be a nonextreme point in $H^\infty_1(\D).$ Also set $\varphi:=b/a$, where $a$ is the Pythagorean mate of $b$. If $\varphi \in H^\infty_1(\D),$ then $X_b$ is admissible.
\end{theorem}
\begin{proof}
Suppose first that $\varphi$ is not a unimodular constant function. Since $\varphi \in H^\infty_1(\D),$ an application of the maximum modulus principle yields that $|\varphi(z)|<1$ for all $z\in \D.$ Since $b=a\varphi$ with $\varphi \in H^\infty_1(\D)$, we have $T_{\bar b}f = T_{\bar a}T_{\bar\varphi}f$ for every $f \in H^2(\D)$. Hence $H(b)=H^2(\D)$ as sets, with $f^+=T_{\bar\varphi}f$ (see \eqref{ffplusrela}) and, by \eqref{inprodhb},
\begin{align*}
    \|f\|_{H^2(\D)}^2 \leq \|f\|_{H(b)}^2 = \|f\|_{H^2(\D)}^2 + \|T_{\bar\varphi}f\|_{H^2(\D)}^2 \leq 2\|f\|_{H^2(\D)}^2 \qquad (f \in H(b)).
\end{align*}
In particular, the $H(b)$-norm and $H^2(\D)$-norm are equivalent, so $\operatorname{span}\{\mathbb S_z : z \in \D\}$ is dense in $H(b)$.

For $0<r\leq 1$, define
\begin{align*}
    \varphi_r(z):=\varphi(rz), \qquad B_r:=T_{\overline{\varphi_r}}, \quad\text{and}\quad B:=B_1=T_{\bar\varphi},
\end{align*}
viewed as operators on $H(b)$, and denote by $B_r^*$ and $B^*$ their adjoints in $\mathcal{B}(H(b))$. We first show that $B_r$ is a contraction on $H(b)$ for every $0<r\leq 1$. Indeed, since $\|\varphi_r\|_\infty\leq 1$, the operator $B_r$ is a contraction on $H^2(\D)$. Moreover, as $B_r B = B B_r$, it follows from \eqref{inprodhb} that for all $f \in H(b)$,
\begin{align*}
    \|B_rf\|_{H(b)}^2
    =\|B_rf\|_{H^2(\D)}^2+\|B(B_rf)\|_{H^2(\D)}^2
    &=\|B_rf\|_{H^2(\D)}^2+\|B_r(Bf)\|_{H^2(\D)}^2\\
    &\leq\|f\|_{H^2(\D)}^2+\|Bf\|_{H^2(\D)}^2
    =\|f\|_{H(b)}^2.
\end{align*}
In particular, $B$ is a contraction on $H(b)$. Since $B\mathbb S_z = \overline{\varphi(z)}\,\mathbb S_z$ and $|\varphi(z)|<1$ for every $z\in\D$, we have $B^n\mathbb S_z \to 0$ in $H(b)$ as $n\to\infty$. Because $\operatorname{span}\{\mathbb S_z : z \in \D\}$ is dense in $H(b)$ and $\sup_{n\in \Z_+}\|B^n\|_{H(b)}\leq 1$, it follows that
\begin{align} \label{puretopa}
    B^n \rightarrow 0 \quad\text{as } n \rightarrow \infty \quad \text{in SOT on } H(b).
\end{align}

Next, since $T_{\bar b}f=T_{\bar a}f^+$, we have $\langle f, b \rangle_{H^2(\D)}=\langle f^+, a \rangle_{H^2(\D)}$ for all $f \in H(b)$. It now follows from \eqref{inprodhb} and \cite[Corollary 23.9]{FM2} that
\begin{align*}
    \langle f, b \rangle_{H(b)}
    =\langle f, b \rangle_{H^2(\D)}+\langle f^+, b^+ \rangle_{H^2(\D)}
    =\langle f^+, a \rangle_{H^2(\D)} +\Big\langle f^+, \frac{1}{a(0)}-a \Big\rangle_{H^2(\D)}
    =\frac{1}{a(0)}f^+(0).
\end{align*}
Therefore, for every $f \in H(b)$, since $f^+=Bf$,
\begin{align*}
    \langle Bf, k_0^b \rangle_{H(b)}=(Bf)(0)=f^+(0)=a(0)\langle f, b \rangle_{H(b)}.
\end{align*}
Thus, $B^*k_0^b=a(0)b.$ Setting $A:=k_0^b\otimes k_0^b$, we have $B^*AB=(B^*k_0^b)\otimes(B^*k_0^b)=a(0)^2(b\otimes b)$, so \eqref{defectD} yields
\begin{align*}
    D=A+B^*AB.
\end{align*}
Consequently, for every $n\in \Z_+$,
\begin{align*}
    \sum_{k=0}^n(-1)^k B^{*k} D B^k = A + (-1)^n B^{*(n+1)} A B^{n+1}.
\end{align*}
Since $\|B^*\|_{H(b)}=\|B\|_{H(b)}\leq 1$, we have $\|B^{*(n+1)}AB^{n+1}f\|_{H(b)} \leq \|A\|_{H(b)}\|B^{n+1}f\|_{H(b)}$ for every $f\in H(b)$, as well as $\sup_{n\in \Z_+}\|B^{*(n+1)}AB^{n+1}\|_{H(b)}\leq \|A\|_{H(b)}$. Combining this with \eqref{puretopa} shows that $B^{*(n+1)}AB^{n+1}\to 0$ in SOT, and hence in UWT by Lemma~\ref{SOT:WOT_UWT}(2). Therefore,
\begin{align*}
    \sum_{k=0}^\infty  (-1)^k B^{*k} D B^k = A \quad\text{in UWT}.
\end{align*}
Since $\mathcal N(X_b)$ is UWT-closed, by Proposition~\ref{necesuffi} and Lemma \ref{nonext: Dmiddle}, it remains to show that $B^{*k} D B^k \in \mathcal N(X_b)$ for each $k\in \Z_+$.

For any polynomials $P$ and $Q$, it is immediate that
\begin{equation}\label{stability}
    P(X_b^*)\,\mathcal N(X_b)\,Q(X_b)\subseteq\mathcal N(X_b).
\end{equation}
For each $z\in\D$, the reproducing property of $\mathbb S_z$ in $H^2(\D)$ yields
\begin{align*}
    B_r\mathbb S_z=\overline{\varphi(rz)}\,\mathbb S_z
    \quad\text{and}\quad
    B\mathbb S_z=\overline{\varphi(z)}\,\mathbb S_z
\end{align*}
(see \eqref{debrangeskernel}). Therefore, $B_r\mathbb S_z\to B\mathbb S_z$ in $H(b)$ as $r\uparrow 1$. Since $\operatorname{span}\{\mathbb S_z : z \in \D\}$ is dense in $H(b)$ and $\sup_{0<r<1}\|B_r\|_{H(b)}\leq 1$, it follows that
\begin{equation}\label{Br-SOT}
    B_r\to B \quad\text{in SOT as } r\uparrow 1.
\end{equation}
For each fixed $k\in \Z_+$, since $\sup_{0<r<1}\|B_r\|_{H(b)}\leq 1$, it follows that
\begin{align*}
    B_r^k\to B^k \quad\text{in SOT}
    \qquad\text{and}\qquad
    B_r^{*k}\to B^{*k} \quad\text{in WOT}
    \qquad\text{as } r\uparrow 1,
\end{align*}
and hence, by Lemma~\ref{SOT:WOT_UWT}(2),
\begin{align} \label{SOTWOT}
    B_r^k\to B^k \quad\text{in UWT}
    \qquad\text{and}\qquad
    B_r^{*k}\to B^{*k} \quad\text{in UWT}
    \qquad\text{as } r\uparrow 1.
\end{align}
Write $\varphi(z)=\sum_{n=0}^\infty  a_n z^n$. For $0<r<1$, since $\sum_{n=0}^\infty  |a_n|r^n < \infty$ and $\|X_b\|_{H(b)}\leq 1$, the series $\sum_{n=0}^\infty  \overline{a}_n r^n X_b^n$ converges in operator norm on $H(b)$, and its sum agrees with $B_r=T_{\overline{\varphi_r}}$ (as $X_b=M_z^*|_{H(b)}$ and the $H(b)$-norm is equivalent to the $H^2(\D)$-norm). Hence, for each $k\in \Z_+$, $B_r^k$ is a norm limit of polynomials in $X_b$, and $B_r^{*k}$ is a norm limit of polynomials in $X_b^*$. Since $D=I-X_b^*X_b\in \mathcal N(X_b)$ (see \eqref{defectD}), it follows from \eqref{stability} and the UWT-closedness of $\mathcal N(X_b)$ that
\begin{align}\label{Brinnxb}
    B_r^{*k} D B_s^k \in \mathcal N(X_b) \qquad (r,s\in(0,1), \; k\in \Z_+).
\end{align}
Now fix $s\in(0,1)$ and let $r\uparrow 1$. By \eqref{SOTWOT},
\begin{align*}
    B_r^{*k} D B_s^k \to B^{*k} D B_s^k \quad\text{in UWT}.
\end{align*}
Since $\mathcal N(X_b)$ is UWT-closed, \eqref{Brinnxb} implies $B^{*k} D B_s^k \in \mathcal N(X_b)$. Letting $s\uparrow 1$ and applying \eqref{SOTWOT} once more gives
\begin{align*}
    B^{*k} D B_s^k \to B^{*k} D B^k \quad\text{in UWT},
\end{align*}
and consequently,
\begin{equation*}
    B^{*k} D B^k \in \mathcal N(X_b) \qquad (k\in \Z_+).
\end{equation*}

Suppose now that $\varphi$ is a unimodular constant function. Since $|a|^2+|b|^2=1$ a.e. on $\mathbb{T}$ and $\varphi=\frac{b}{a}$,
both $a$ and $b$ are constants with $|b|=1/\sqrt{2}<1$. Consequently, $H(b)$ coincides with $H^2(\D)$, with its norm differing only by a constant factor ($\|f\|_{H(b)}=\sqrt{2}\,\|f\|_{H^2(\D)}$). Moreover, $X_b$ is unitarily equivalent to $M_z^*$. Hence, by Lemma \ref{admtwoproperties}, $X_b$ is admissible. This completes the proof.
\end{proof}

In view of the above results, we have the following result as an application of Theorem \ref{mainthm}.
\begin{theorem} \label{Xb_nonextreme_result}
In Setup~\ref{setup}, let $\clh_i = H(b_i)$ for nonextreme points $b_1,\ldots,b_n \in H_1^\infty(\D)$, and for each $i=1,\ldots,n$, suppose either
\begin{enumerate}
    \item $T_i = X_{b_i}$ with $b_i/a_i \in H_1^\infty(\D)$ (where $a_i$ is the Pythagorean mate of $b_i$), or
    \item $T_i = X_{b_i}^*$ with $b_i$ cyclic for $M_z^*$ on $H^2(\D)$.
\end{enumerate}
Then a closed jointly $(\tilde T_1,\ldots,\tilde T_n)$-invariant subspace $\fM \subseteq \fH$ is doubly commuting if and only if it is hereditary.
\end{theorem}
We end this section with the following broad class of examples of nonextreme points $b$ satisfying the hypotheses of Theorem \ref{non-extre adm} and Theorem \ref{Xbstaradm}.
\begin{example} \label{nonex adm ex}
Take nonextreme point $b\in H^\infty_1(\D)$ such that  $\|b\|_{\infty}\leq \frac{1}{\sqrt 2}.$ 
From $|a|^2=1-|b|^2$ a.e. on $\mathbb T,$
we get
\[
\left|\frac {b}{a}\right|^2=\frac{|b|^2}{1-|b|^2}\leq1 \qquad\text{a.e. on }\mathbb T.
\]
Since $a$ is outer, it follows from \cite[Theorem 2.11]{D1970} that
\[
    \frac ba\in H_1^\infty(\D).
\]
So, $b$ satisfies the hypotheses of Theorem \ref{non-extre adm}. For any non-rational and nonextreme point $b\in H^\infty_1(\D)$ that is holomorphic in a neighborhood of $\overline \D,$ by \cite[Theorem 2.2.4]{DSS1970}, it satisfies the hypothesis of Theorem \ref{Xbstaradm}.
\end{example}

\section{Applications to weighted Bergman and complete Nevanlinna-Pick spaces}\label{CNP spaces}
In this section, we show that the adjoint of the shift operator $M_z^*$ on the weighted Bergman space $A^2_{\omega}(\D)$, associated with a weight sequence specified below, is admissible. This result enables us to characterize doubly commuting quotient modules of the weighted Bergman space $A^2_{\bm{\omega}}(\D^n)$, associated with a suitable multi-weight sequence $\bm{\omega}$. Furthermore, in the setting of a complete Nevanlinna–Pick space $\clh(K)$ associated with a complete Nevanlinna–Pick kernel $K$, we prove that the adjoint of the shift operator $M_z^*$ on $\clh(K)$ is admissible. This, in turn, allows us to characterize doubly commuting quotient modules of tensor products of certain complete Nevanlinna–Pick spaces.
\subsection{Doubly commuting quotient modules of weighted Bergman spaces
\texorpdfstring{$A^2_{\texorpdfstring{\bm{\omega}}{\omega}}(\mathbb{D}^n)$}{A2(\mathbb{D}n)}} \label{weighted bergman}
Let $\omega=(\omega_m)_{m\in \Z_+}$ be a positive non-increasing sequence such that 
\begin{align} \label{weightsequ}
    \omega_0=1 \quad \text{and}\quad \liminf_{m \to \infty} \omega_m^{\frac{1}{m}} = 1.
\end{align}
Corresponding to the sequence $\omega$, consider the weighted Bergman space over $\D$, denoted by $A^2_{\omega}(\D)$, and defined as 
\[A_\omega^2(\D):= \left \{f(z) = \sum_{m =0}^\infty  a_mz^m:  \sum_{m =0}^\infty  |a_m|^2 \omega_m < \infty, \,a_m \in \C,\, z\in \D\right \}.\]
It is a reproducing kernel Hilbert space (for an exposition, see \cite{PR}) with the kernel  
\[ K_{\omega} : \D \times \D \to \C\qquad \text{ defined by}\quad K_{\omega}(\zeta,\eta) = \sum_{m =0}^\infty  \frac{1}{\omega_m} 
(\zeta\bar{\eta})^m   \quad (\zeta,\eta \in \D).\]
The kernel $K_{\omega}$ has an associated analytic function $h_{\omega}$ on $\D$ defined by 
\begin{align}\label{ana_fun_omega}
    h_{\omega}(z)=\sum_{m =0}^\infty  \frac{1}{\omega_m} 
z^m \quad (z\in \D).\end{align} 
The multiplication operator $M_z$, known as shift operator, on $A^2_{\omega}(\D)$ is defined by $(M_z f)(\eta)= \eta f(\eta)$ for all $f \in A^2_{\omega}(\D)$ and $\eta \in \D$. Note that $\{\frac{z^m}{\sqrt{\omega_m}}\}_{m\in \Z_+}$ is an orthonormal basis of $A^2_\omega(\D)$.  Denote $e_m=\frac{z^m}{\sqrt{\omega_m}}$ for $m \in \Z_+$. An easy computation shows that
\begin{align}\label{action_weighted_shift}
    M_z e_m=\sqrt{\frac{\omega_{m+1}}{\omega_{m}}}\,e_{m+1},
\qquad
M_z^*e_m=
\begin{cases}
\sqrt{\frac{\omega_{m}}{\omega_{m-1}}}\,e_{m-1},&m\geq1,\\
0,&m=0.
\end{cases}
\end{align}
This shows that 
\[ \|M_z\|_{\omega}^2 = \sup_m \frac{\omega_{m+1}}{\omega_m} .\]
Since $\{\omega_m\}_{m\ge 0}$ is a positive decreasing sequence, it follows that $M_z$ is a contraction (see \eqref{weightsequ}). Throughout the remainder of this subsection, we assume that a weight sequence $\omega$ satisfies the following properties:
\begin{hypothesis}\label{admweightseq}
The sequence $\omega=(\omega_m)_{m\ge 0}$ is positive and decreasing with $\omega_0=1$ and $\liminf_{m\to\infty}\omega_m^{\frac{1}{m}}=1$. Also, assume that the corresponding analytic function $h_\omega$ as in \eqref{ana_fun_omega} is non-vanishing on $\D$, and $\frac{h_\omega}{h_{\omega,r}}$ has nonnegative Taylor coefficients for every $0<r<1$, where $h_{\omega,r}(z)=h_\omega(rz)$.
\end{hypothesis}
Some natural examples of weight sequences satisfying Hypothesis \ref{admweightseq} are given below. The first is the constant sequence $\omega_m=1$ for all $m\in \Z_+$. The second is the sequence $\omega^{(\alpha)}=(\omega_m^{(\alpha)})_{m\in \Z_+}$, where $\alpha>1$ is fixed and
\begin{align} \label{weightedbergseq}
    \omega_m^{(\alpha)} = \frac{1}{\binom{\alpha+m-1}{m}} \quad \text{and}\quad \binom{\alpha+m-1} {m}=\frac{\Gamma(\alpha+m)}{\Gamma{(\alpha){m!}}}, \quad m\in \Z_+.
\end{align}
The constant sequence case corresponds to the Hardy space $H^2(\D)$, while the latter corresponds to the weighted Bergman space over $\D$ with kernel  
 $K_\alpha(z,w) = \frac{1}{(1-z\bar{w})^\alpha}$ ($z,w\in\D$).
 
 Now we fix a weight sequence $\omega$ satisfying Hypothesis \ref{admweightseq}. Since $h_\omega$ is non-vanishing, 
 \begin{equation*}
    \frac{1}{h_\omega(z)}=\sum_{l\in \Z_+}c_l z^l\quad(z\in \D),
\end{equation*}
for some $c_l \in \C.$ For every $0<r<1$,
\begin{equation*}
\left(\sum_{m =0}^\infty  \frac{1}{\omega_m} 
z^m\right)\left(\sum_{l=0}^\infty  c_lr^l z^l\right)=\frac{h_\omega (z)}{h_{\omega,r}(z)}=\frac{h_\omega(z)}{h_\omega(rz)}=:\sum_{m\ge0}a_m(1, r)z^m \quad (z \in \D).
\end{equation*}
Comparing the coefficients of last identity, we have $a_0(1,r)=1$, and 
\begin{equation}\label{eq_a_c}
    a_m(1, r) =\sum_{j=0}^m \frac{1}{\omega_{m-j}} c_{j} r^{j}, \quad (m\geq 1).
\end{equation} 
On the other hand, the identity $h_\omega(z)=h_\omega(rz)\sum_{m=0}^\infty  a_m(1,r)z^m$ implies
\begin{align}\label{eq_a_m}
    \frac{1}{\omega_m}=\sum_{j=0}^m \frac{1}{\omega_j}r^j a_{m-j}(1, r)\geq a_m(1, r) \qquad (m \in \Z_+),
\end{align}
because every summand is nonnegative by Hypothesis~\ref{admweightseq} and $\omega_0=1$. Since $1/h_\omega(z)=\sum_{l=0}^\infty  c_l z^l$ is analytic on $\D$, we have $\sum_{l=0}^\infty |c_l|r^l<\infty$ for every $0<r<1$. As $M_z$ is a contraction, $\|r^l c_l M_z^l M_z^{*l}\|\leq |c_l|r^l$, so the series
\[
    D_{\omega,M^*_z}(r):=\sum_{l=0}^\infty r^l c_l M_z^l M_z^{*l}
\]
converges in operator norm on $A_\omega^2(\D)$. In particular, $D_{\omega,M^*_z}(r)\in\cln(M^*_z)$ for each $0<r<1$. From \eqref{action_weighted_shift}, a simple calculation gives
\[
M_z^lM_z^{*l}e_m=
\begin{cases}
\frac{\omega_{m}}{\omega_{m-l}}\,e_m,&l\leq m,\\
0,&l>m.
\end{cases}
\]
Therefore, by \eqref{eq_a_c} (together with $a_0(1,r)=1=c_0/\omega_0$),
\[
D_{\omega, M^*_z}(r)e_m=\omega_m\sum_{l=0}^m c_l r^l \frac{1}{\omega_{m-l}}\,e_m
=a_m(1,r)\omega_m e_m \qquad (m \in \Z_+).
\]
Thus, for any $f(z)=\sum_{m=0}^{\infty}b_m z^m \in A^2_\omega(\D)$, we have
\begin{align}\label{action_D}
    D_{\omega, M^*_z}(r)f(z)
    =\sum_{m=0}^{\infty} a_m(1, r)\omega_m b_m z^m.
\end{align}
Now, using Hypothesis~\ref{admweightseq} and \eqref{eq_a_m}, we obtain
\[
0\leq
\langle D_{\omega, M^*_z}(r)f,f\rangle_{\omega}
=\sum_{m=0}^{\infty} a_m(1, r)\omega^{2}_m|b_m|^2
\leq
\sum_{m=0}^{\infty}|b_m|^2 \omega_m
=
\|f\|^2_\omega.
\]
This shows that $0\leq D_{\omega, M^*_z}(r)\leq I$, and hence $\|D_{\omega, M^*_z}(r)\|\leq 1$, for all $0<r<1$. On the other hand, for each fixed $m\geq 1$, \eqref{eq_a_c} implies that $\lim_{r\to 1^-} a_m(1,r)=\sum_{j=0}^m \frac{c_j}{\omega_{m-j}}$, which is the $m$-th Taylor coefficient of $h_\omega(z)\cdot\frac{1}{h_\omega(z)}=1$ and hence vanishes; thus $a_m(1,r)\to 0$ as $r\to 1^-$. Moreover, by \eqref{eq_a_m},
\[
|a_m(1,r)|^2 \omega_m^{3}|b_m|^2 \leq |b_m|^2 \omega_m \qquad (m \geq 1),
\]
with $\sum_{m=1}^{\infty}|b_m|^2\omega_m \leq \|f\|_{\omega}^2 < \infty$. Since $a_0(1,r)\omega_0=1$, it follows from \eqref{action_D} and the dominated convergence theorem that
\[
\|(D_{\omega, M^*_z}(r)-P_{\C \bm 1})f\|_{\omega}^2
=\sum_{m=1}^{\infty} |a_m(1, r)|^2\omega^3_m |b_m|^2 \to 0 \quad \text{as } r\to 1^-.
\]
Hence $\lim_{r\to 1^-} D_{\omega,M^*_z}(r)= P_{\C \bm 1}$ in SOT. Since $\sup_{0<r<1}\|D_{\omega,M^*_z}(r)\|\leq 1$, Lemma~\ref{SOT:WOT_UWT}(2) yields $\lim_{r\to 1^-} D_{\omega,M^*_z}(r)= P_{\C \bm 1}$ in UWT. Because $D_{\omega,M^*_z}(r)\in\cln(M^*_z)$ for every $0<r<1$ and $\cln(M^*_z)$ is UWT-closed, we conclude that $P_{\C \bm 1}\in \cln(M^*_z)$. Finally, $\|\bm 1\|_\omega=\sqrt{\omega_0}=1$, and since $\omega_k>0$ for all $k\in \Z_+$, we have
\[
\overline{\operatorname{span}}\{M_z^{k}\bm 1 : k\in \Z_+\}=A_\omega^2(\D).
\]
Therefore, by Proposition~\ref{cyclic_vector}, we obtain the following result.
\begin{proposition}
For a weight sequence $\omega$ satisfying Hypothesis \ref{admweightseq}, the operator $M^*_z$ is an admissible operator on $A^2_\omega(\D)$.
\end{proposition}
For each $i=1,\ldots,n$, let $\omega^{(i)}=(\omega_{m}^{(i)})_{m \in \Z_+}$ be a weight sequence satisfying Hypothesis \ref{admweightseq}, and consider the multi-weight sequence $\bm{\omega} = (\omega^{(1)},\ldots, \omega^{(n)})$. Corresponding to $\bm{\omega}$, consider the weighted Bergman space $A^2_{\bm{\omega}}(\D^n)$ over $\D^n$, defined as
\[
A^2_{\bm{\omega}}(\D^n):=\Bigg\{f(\z)=\sum_{\bm{k}\in \mathbb{Z}^n_{+}}a_{\bm{k}}\z^{\bm{k}}:
\sum_{\bm{k}\in\mathbb{Z}^n_{+}}|a_{\bm{k}}|^2\omega_{k_1}^{(1)}\cdots\omega_{k_n}^{(n)}<\infty,\, a_{\bm{k}}\in\C,\, \z\in\D^n\Bigg\}.
\]
The Hilbert space $A^2_{\bm{\omega}}(\D^n)$ is unitarily equivalent to $A^2_{\omega^{(1)}}(\D)\otimes\cdots\otimes A^2_{\omega^{(n)}}(\D)$ via the natural unitary operator that sends the monomials $\z^{\bm{k}}$ to $z_1^{k_1}\otimes \cdots \otimes z_n^{k_n}$ for all $\bm{k}\in\mathbb{Z}_{+}^n$. We refer the reader to \cite{O, BDDP} for a comprehensive treatment of such weighted Bergman spaces.

Now we set $\clh_i=A^2_{\omega^{(i)}}(\D)$ and $T_i=M_z^*$ in Setup \ref{setup}. Then
\[
M_{z_i}^*:=\tilde{T}_i=\big(I_{ A^2_{\omega^{(1)}}(\D)}\otimes\cdots\otimes \underbrace{M_z^*}\limits_{i\textup{-th place}}\otimes\cdots \otimes I_{ A^2_{\omega^{(n)}}(\D)}\big)\quad (i=1,\ldots,n).
\]
Let $\fQ\subseteq A^2_{\bm\omega}(\D^n)$ be a closed subspace. We say that $\fQ$ is a \textit{quotient module of $A^2_{\bm{\omega}}(\D^n)$} if $\fQ$ is jointly $(M_{z_1}^*,\ldots,M_{z_n}^*)$-invariant. The following result is an immediate application of Theorem \ref{mainthm} and, in the special case where each $\omega^{(i)}=\omega^{(\alpha_i)}$ is of the form \eqref{weightedbergseq} (with $\alpha \in \mathbb{N}$), recovers \cite[Theorem 4.9]{CDS}.

\begin{theorem}\label{resultweightedberg}
Let $\bm{\omega}=(\omega^{(1)},\ldots,\omega^{(n)})$ be a multi-weight sequence such that each $\omega^{(i)}$ ($i=1,\ldots,n$) satisfies Hypothesis \ref{admweightseq}. Then a quotient module $\fQ$ of $A^2_{\bm{\omega}}(\D^n)$ is doubly commuting if and only if
\[
\fQ = \clq_1 \otimes \cdots \otimes \clq_n
\]
for some quotient modules $\clq_i$ of $A^2_{\omega_i}(\D)$ for all $i=1,\ldots,n$.
\end{theorem}

\subsection{Doubly commuting quotient modules of tensor products of complete Nevanlinna-Pick spaces} \label{NPS}
Let $K:\D\times\D\to\C$ be a reproducing kernel of the form
\begin{align}\label{CNP_Kernel}
K(z,w)=\sum_{l=0}^{\infty}a_l(z\overline{w})^l,
    \qquad a_0=1,\quad a_l>0,
\end{align}
where the series converges for all $(z,w)\in\D\times\D$. We assume that $K$ is a complete Nevanlinna-Pick kernel on $\D$. The reproducing kernel Hilbert space corresponding to $K$ is denoted by $\clh(K)$. It is well-known that for a kernel as in \eqref{CNP_Kernel} the coordinate multiplier $M_z$ on $\clh(K)$, defined by $M_zf(w)=wf(w)$ for $f\in \clh(K)$ and $w\in \D$, is bounded if and only if $\sup_{l \in \Z_+} a_l/a_{l+1}<\infty$.
It is easy to see that with respect to the orthonormal basis $\{e_l\}_{l\in \Z_+}$, the shift $M_z$ satisfies
\[
M_z e_l=\sqrt{\frac{a_l}{a_{l+1}}}\,e_{l+1},
\qquad
M_z^*e_l=
\begin{cases}
\sqrt{\frac{a_{l-1}}{a_l}}\,e_{l-1},&l\geq1,\\
0,&l=0,
\end{cases}
\]
where $e_l=\sqrt{a_l}\,z^l$ for $l\in \Z_+$. Moreover, by iterating the above identities, we obtain
\begin{align}\label{action_shifts}
    M_z^j M_z^{*j}e_l=
\begin{cases}
(a_{l-j}/a_l)e_l,&j\leq l,\\
0,&j>l.
\end{cases}
\end{align}
 Examples of complete Nevanlinna-Pick kernels satisfying \eqref{CNP_Kernel} include the reproducing kernel of the Hardy space $H^2(\D)$ and the reproducing kernel of the classical Dirichlet space over $\D$ which is
\[
K(z,w)=\sum_{n=0}^{\infty}\frac{(z\overline w)^n}{n+1}\quad(z,w\,\in\D).
\] 
Observe that the weight sequence given by $\omega_m=m+1$ does not fulfill Hypothesis \ref{admweightseq}. In the following result, we prove that the backward shift operator $M^*_z$ on $\clh(K)$ is an admissible operator. 
\begin{proposition}\label{gerneral_RKHS}
Let $K$ be a complete Nevanlinna-Pick kernel satisfying \eqref{CNP_Kernel} such that $M_z$ is bounded. Then $M^*_z$ on $\clh(K)$ is an admissible operator.
\end{proposition}
\begin{proof}
Since $K$ is a complete Nevanlinna--Pick kernel of the form \eqref{CNP_Kernel} (so $K(0, w)=1$ for all $w\in\D$), by \cite[Theorem 7.33]{AglerM} there exist $b_j\geq 0$ for all $j\geq 1$ such that
\begin{align*}
    1-\frac{1}{\sum_{l=0}^\infty a_l (z\bar{w})^l}=\sum_{j\geq 1}b_j (z\bar{w})^j \qquad (z, w\in \D).
\end{align*}
Multiplying both sides by $\sum_{l=0}^\infty a_l (z\bar{w})^l$ and comparing coefficients, we obtain
\begin{equation}\label{eq:recurrence}
a_l=\sum_{j=1}^{l}b_j a_{l-j} \qquad (l\geq 1).
\end{equation}
For each $N\geq 1$, define
\[
F_N:=I-\sum_{j=1}^{N}b_j M_z^j M_z^{*j}.
\]
Clearly, $F_N\in\cln(M_z^*)$ for every $N\geq 1$. We claim that $F_N \to P_{\C \bm 1}$ in UWT as $N\to\infty$. To prove the claim, let $f(z)=\sum_{l=0}^\infty  c_l z^l\in \clh(K)$. Using \eqref{action_shifts}, we obtain
\begin{align*}
(F_N f)(z)
=f(z)-\sum_{j=1}^{N}b_j(M_z^j M_z^{*j}f)(z)
=c_0+\sum_{l=1}^{\infty}c_l
\left(1-\frac{1}{a_l}
\sum_{j=1}^{\min\{N,l\}}b_j a_{l-j}\right)z^l.
\end{align*}
By \eqref{eq:recurrence}, we have
\[
1-\frac{1}{a_l}\sum_{j=1}^{l}b_j a_{l-j}=0 \qquad (1\leq l\leq N),
\]
and for $l>N$,
\[
r_{N,l}:=1-\frac{1}{a_l}\sum_{j=1}^{N}b_j a_{l-j}
=\frac{1}{a_l}\sum_{j=N+1}^{l}b_j a_{l-j}.
\]
Thus,
\begin{align}\label{eq:F_n}
    (F_N f)(z)=c_0+\sum_{l>N}c_l r_{N,l}z^l.
\end{align}
Since $b_j\geq 0$ for all $j\geq 1$ and $a_k>0$ for all $k\in\Z_+$, \eqref{eq:recurrence} implies that
\[
0\leq\sum_{j=N+1}^{l}b_j a_{l-j}
\leq\sum_{j=1}^{l}b_j a_{l-j}=a_l,
\]
and hence $0\leq r_{N,l}\leq 1$ for all $l>N$. In particular, \eqref{eq:F_n} shows that $F_N$ is diagonal with respect to the orthonormal basis $\{e_l\}_{l\in \Z_+}$ with eigenvalues in $[0,1]$, so $0\leq F_N\leq I$ and $\|F_N\|\leq 1$ for all $N\geq 1$.

Since $a_0=1$, the constant function $\bm 1=e_0$ is a unit vector orthogonal to $\{z^l : l\geq 1\}$, and therefore $(P_{\C \bm 1}f)(z)=c_0$. Using \eqref{eq:F_n} and $\|f\|_{\clh(K)}^2=\sum_{l=0}^\infty |c_l|^2/a_l<\infty$, we obtain
\begin{align*}
\|F_N f-P_{\C \bm 1}f\|_{\clh(K)}^2
=\sum_{l>N}\frac{|c_l|^2}{a_l}r_{N,l}^2
\leq\sum_{l>N}\frac{|c_l|^2}{a_l}
\to 0 \qquad \text{as } N\to\infty.
\end{align*}
This shows that $F_N\to P_{\C \bm 1}$ in SOT as $N\to\infty$. Since $\sup_{N\geq 1}\|F_N\|\leq 1$, Lemma~\ref{SOT:WOT_UWT}(2) implies that $F_N\to P_{\C \bm 1}$ in UWT, establishing the claim. Because $F_N\in\cln(M_z^*)$ for all $N\geq 1$ and $\cln(M_z^*)$ is UWT-closed, it follows that $P_{\C \bm 1}\in \cln(M_z^*)$. Finally, since $\bm 1$ is a unit vector and $a_k>0$ for all $k\in \Z_+$, we have
\[
\overline{\operatorname{span}}\{M_z^{k} \bm 1 : k\in \Z_+\}=\clh(K).
\]
By Proposition~\ref{cyclic_vector}, $M_z^*$ is an admissible operator. This completes the proof.
\end{proof}
Now, for $i = 1, \dots, n$, let $K_{i}$ be a complete Nevanlinna--Pick kernel on $\mathbb{D}$ satisfying \ref{CNP_Kernel} such that $M_{z}$ is bounded on $\mathcal{H}(K_{i})$. A closed $M_{z}^{*}$-invariant subspace of $\mathcal{H}(K_{i})$ is called a quotient module of $\mathcal{H}(K_{i})$. With the notation of Setup \ref{setup}, setting $\mathcal{H}_{i} = \mathcal{H}(K_{i})$ and $T_{i} = M_{z}^{*}$ for each $i = 1, \dots, n$, we call a closed jointly $(\tilde{T}_{1}, \dots, \tilde{T}_{n})$-invariant subspace of $\mathcal{H}(K_{1}) \otimes \cdots \otimes \mathcal{H}(K_{n})$ a quotient module of $\mathcal{H}(K_{1}) \otimes \cdots \otimes \mathcal{H}(K_{n})$. We then obtain the following result as a consequence of Theorem \ref{mainthm}.

\begin{theorem}\label{resultNPS}
For each $i=1,\ldots,n$, let $K_i$ be a complete Nevanlinna-Pick kernel on $\D$ satisfying \eqref{CNP_Kernel} such that $M_z$ is bounded on $\clh(K_i)$. A closed subspace $\fQ$ of $\clh(K_1)\otimes\cdots\otimes\clh(K_n)$ is a doubly commuting quotient module if and only if
\[
\fQ=\clq_1\otimes\cdots\otimes\clq_n,
\]
for some quotient modules $\clq_i$ of $\clh(K_i)$, for all $i=1,\ldots,n$.
\end{theorem}
\section{Applications to coordinate Volterra operators on \texorpdfstring{$L^2\big((0,1)^n;\C\big)$}{L2((0,1)n; C)}}\label{Vol}
Recall that Volterra operator on $L^2((0,1), \C)$ is defined as 
$$Vf(x)=\int_{0}^xf(t)dt \qquad(0<x<1).$$

By Sarason's theorem \cite[Theorem 1]{S65}, $(I+V)^{-1}$ is unitarily equivalent to $(I+P_\clq M_z|_\clq)/2$, where 
\[
\clq^\perp =\exp\left(-\frac{1+z}{1-z}\right)H^2(\D).
\]
Since $P_\clq M_z|_\clq$ is admissible, so is $(I+V)^{-1}$ (see Propostion \ref{extxbxbstar}). Moreover, as $\|V\|<1$, by Lemma \ref{admtwoproperties}, $V$ is admissible.
Since $M_z^*|_\clq$ is also admissible, $V^*$ is also admissible. The preceding discussion yields the following.
\begin{proposition}\label{Volterra_admi}
Both $V$ and $V^*$ are admissible operators.
\end{proposition}
Let $n\geq1$ and $0\leq a_i<b_i$ be real numbers for all $i=1,\ldots,n$. Write 
\[
\prod_{i=1}^n(a_i,b_i):=(a_1,b_1)\times\cdots\times(a_n,b_n)\quad\text{and}\quad (0,1)^n:=\prod_{i=1}^n(0,1).
\]
Let $L^2\big(\prod_{i=1}^n(a_i,b_i);\C\big)$ be the Hilbert space of square-integrable complex-valued functions on $\prod_{i=1}^n(a_i,b_i)$, equipped with the usual inner product
$$\langle f ,g\rangle =\int_{\prod_{i=1}^n(a_i,b_i)}f(x)\overline{g(x)}\,dx.$$ 
By the standard tensor-product identification for $L^2$-spaces,
\begin{equation}\label{L2_id}
L^2((a_1,b_1);\C) \otimes \cdots \otimes L^2((a_n,b_n);\C)\cong L^2\Big(\prod_{i=1}^n(a_i,b_i);\C\Big),
\end{equation}
via the identification map
\begin{equation}\label{L2_id_map}
f_1\otimes\cdots\otimes f_n\mapsto f_1\cdots f_n\qquad(f_i\in L^2((a_i,b_i);\C), \,i=1,\ldots,n).
\end{equation}
For $i=1, \ldots, n$, define the coordinate Volterra operator
$V_i$ on $L^2\big((0,1)^n;\C\big)$ by
\begin{equation*}
(V_i f)(x_1,\ldots,x_n)
=\int_0^{x_i}
 f(x_1,\ldots,x_{i-1},t,x_{i+1},\ldots,x_n)\,dt.
\end{equation*}
These operators are studied in \cite[Theorem 2]{K1966}. It can be shown using \eqref{L2_id} and \eqref{L2_id_map} that 
\begin{align*}
V_i
\cong I_{L^2((0,1);\C)}\otimes\dots\otimes \underbrace{ V }\limits_{i\textup{-th place}}\otimes \dots\otimes I_{L^2((0,1);\C)}.
\end{align*}
Since $V$ and $V^*$ are admissible by Proposition \ref{Volterra_admi}, in compliance with \cite[Theorem 2]{S65}, Theorem \ref{mainthm} gives the following characterization of doubly commuting jointly $(S_1,\ldots,S_n)$-invariant subspaces of $L^2\big((0,1)^n;\C\big)$, where $S_i$ is either
$V_i$ or $V_i^*$. 
\begin{theorem}
Let $(S_1,\ldots,S_n)$, where each $S_i$ is either
$V_i$ or $V_i^*$. A closed subspace $\fM$ of $L^2\big((0,1)^n;\C\big)$ is a
doubly commuting $(S_1,\ldots,S_n)$-invariant subspace if and only if
there exist $a_1,\ldots,a_n\in[0,1]$ such that
\[
\fM=\clm_1\otimes\cdots\otimes\clm_n,
\]
where
\[
\clm_i=
\begin{cases}
L^2(a_i,1), & \text{if } S_i=V_i,\\[2mm]
L^2(0,a_i), & \text{if } S_i=V_i^*.
\end{cases}
\]
\end{theorem}

Using the identification \eqref{L2_id}, we immediately have the following corollary of the above theorem.
\begin{corollary}
The doubly commuting closed jointly invariant subspaces of the operator tuples
$(V_1,\ldots,V_n)$ and $(V_1^*,\ldots,V_n^*)$ on $L^2\big((0,1)^n;\C\big)$
are precisely the subspaces of the forms
\[
L^2\Big(\prod_{i=1}^n(a_i,1);\C\Big)
\quad\text{and}\quad
L^2\Big(\prod_{i=1}^n(0,a_i);\C\Big),
\]
respectively, where $a_1,\ldots,a_n\in[0,1]$.
\end{corollary}
\section{Concluding remarks}
This paper investigates the structure of doubly commuting invariant subspaces under certain tuples of operators acting on tensor product spaces. We unify several existing results within a common operator-theoretic framework and establish Theorem \ref{mainthm} in a general abstract setting. In the applications considered in the preceding sections, the building blocks of the tensor product spaces and the corresponding operator tuples acting on them were drawn one class at a time from the following families of function spaces and associated admissible operators: 
\begin{itemize}
\item[(I)] the Hardy space $H^2(\D)$ and the operator $M_z^*$;
\item[(II)] nonzero quotient modules $\clq_{\theta}$ and the operator $S_{\theta}$;
\item[(III)] de Branges-Rovnyak spaces $H(b)$ for nonconstant extreme points $b\in H^\infty_1(\D)$ and the operators $X_b$ or $X_b^*$;
\item[(IV)] de Branges-Rovnyak spaces $H(b)$ for nonextreme points $b\in H^\infty_1(\D)$ such that $b/a \in H^\infty_1(\D)$, where $a$ is the Pythagorean mate of $b$, and the operator $X_b$ or $H(b)$ for nonextreme points $b\in H^\infty_1(\D)$ such that $b$ is a cyclic vector for $M_z^*$ on $H^2(\D)$ and the operator $X_b^*$;
\item[(V)] weighted Bergman spaces $A^2_\omega(\D)$ associated with weight sequences $\omega$ satisfying Hypothesis \ref{admweightseq} and the operator $M_z^*$;
\item[(VI)] reproducing kernel Hilbert spaces $\clh(K)$ corresponding to complete Nevanlinna-Pick kernels $K$ on $\D$ satisfying \eqref{CNP_Kernel} (such that $M_z$ is bounded on $\clh(K)$) and the operator $M_z^*$; 
\item[(VII)] The space $L^2((0,1);\C)$ and Volterra operator or its adjoint.
\end{itemize}

More generally, one may consider mixed tensor product spaces and the corresponding mixed operator tuples whose building blocks are chosen from any combination of the classes listed above, and characterize their doubly commuting invariant subspaces in hereditary form. More precisely, we state the following.

\begin{theorem} \label{mixedtheorem}
Let the pairs of Hilbert spaces and associated operators $(\clh_i,T_i)$, for $i=1,\ldots,n$, be chosen from any combination of the classes \textup{(I)-(VII)} listed above. Let $\fM$ be a closed jointly invariant subspace of $\fH$ under $(\tilde T_1, \ldots, \tilde T_n)$ defined as in Setup \ref{setup}. Then $\fM$ is doubly commuting if and only if $\fM$ is hereditary.
\end{theorem}

Note that the proof of Theorem \ref{mixedtheorem} follows immediately from Theorem \ref{mainthm}, since the operators $T_i\in\clb(\clh_i)$ considered in the theorem are admissible for all $i=1,\ldots,n$. 

We conclude this article with the following question that arises naturally from our analysis in light of Theorem \ref{mainthm} and Proposition \ref{dctosi}. 

\begin{question} \label{n-1adm}
Let $\clh_i, T_i, \tilde T_i$ ($i=1, \ldots, n$) and $\fH$ be as in Setup \ref{setup} with $\dim \clh_i > 1$ for each $i$. Assume that every doubly commuting jointly $(\tilde T_1, \ldots, \tilde T_n)$-invariant subspace $\mathfrak{M}$ of $\mathfrak{H}$ is hereditary. Must at least $n-1$ of the operators $T_1, \ldots, T_n$ be admissible?
\end{question}

The assumption $\dim\mathcal H_i>1$ for $i=1,\ldots,n$
in Question \ref{n-1adm} entails no loss of generality,
since every nonzero operator on a one-dimensional Hilbert space
is admissible, and such tensor factors can simply be
omitted.

\vspace{0.1in} \noindent\textbf{Acknowledgement:} The research of the first-named author is supported by DST-INSPIRE Faculty Fellowship No. DST/INSPIRE/Faculty/2023/IFA-23-MA-201. The second-named author's research is supported by NBHM Postdoctoral Fellowship No. 0204/37(2)/NBHM/R\\\&D-II/16542. The third-named author's research is supported by NBHM Postdoctoral Fellowship No. 0204/27/(26)/2023/R\&D- II/11926.

\end{document}